\documentclass[reqno,11pt]{amsart}
\usepackage{amssymb, amsmath,latexsym,amsfonts,amsbsy, amsthm}
\usepackage{color}

\let\ve=\varepsilon

\newcommand{\beq}{\begin{equation}}
\newcommand{\eeq}{\end{equation}}
\newcommand{\ben}{\begin{eqnarray}}
\newcommand{\een}{\end{eqnarray}}
\newcommand{\beno}{\begin{eqnarray*}}
\newcommand{\eeno}{\end{eqnarray*}}

\renewcommand{\theequation}{\thesection.\arabic{equation}}

\newtheorem{Theorem}{Theorem}[section]

\newtheorem{Proposition}[Theorem]{Proposition}
\newtheorem{Lemma}[Theorem]{Lemma}
\newtheorem{Corollary}[Theorem]{Corollary}
\newtheorem{Remark}[Theorem]{Remark}

\newcommand{\nn}{\mathbf{n}}

\begin{document}
\title[ Sharp interface limit of a diffuse interface  model for  tumor-growth]
{Sharp interface limit of a diffuse interface  model for  tumor-growth}

\author{Mingwen Fei}
\address{School of  Mathematics and Computer Sciences, Anhui Normal University, Wuhu 241002, China}
\email{ahnufmwen@126.com}

\author{Tao Tao}
\address{School of  Mathematical Sciences, Peking University, Beijing 100871, China}
\email{taotao@amss.ac.cn}

\author{Wei Wang}
\address{Department of Mathematics, Zhejiang University, Hangzhou 310027, China}
\email{wangw07@zju.edu.cn}


\date{\today}
\maketitle

\renewcommand{\theequation}{\thesection.\arabic{equation}}
\setcounter{equation}{0}
\begin{abstract}
we  consider the asymptotic  limit of  a diffuse interface  model for   tumor-growth when a parameter $\varepsilon$ proportional to the thickness of the
diffuse interface goes to zero. An approximate solution which shows explicitly the behavior of the true solution for small $\varepsilon$ will be constructed by using matched expansion method. Based on the energy method, a spectral condition in particular, we establish a smallness estimate of the difference between the approximate solution and the true solution.\end{abstract}

\numberwithin{equation}{section}

\indent

\section{Introduction}

\indent

In this paper we consider the  singular limit, as $\varepsilon\rightarrow0$, of the solutions of the following system for $(u^\varepsilon,\sigma^\varepsilon)$:
\begin{eqnarray}
\left \{
\begin {array}{ll}
u^\varepsilon_t-\Delta \mu^\varepsilon=2\sigma^\varepsilon+u^\varepsilon-\mu^\varepsilon,\ \qquad\qquad\qquad&\text{in}\ \ \Omega\times(0,T),\\[3pt]
\sigma^\varepsilon_t-\Delta \sigma^\varepsilon=-(2\sigma^\varepsilon+u^\varepsilon-\mu^\varepsilon),\ &\text{in}\ \ \Omega\times(0,T),\\[3pt]
\varepsilon\mu^\varepsilon=-\varepsilon^2\Delta u^\varepsilon+f'(u^\varepsilon),\ &\text{in}\ \ \Omega\times(0,T),\\[3pt]
u^\varepsilon(x,0)= u^\varepsilon_0(x),\ \ \sigma^\varepsilon(x,0)= \sigma^\varepsilon_0(x),&\text{on}\ \ \Omega\times\{0\},\\[3pt]
\frac{\partial u^\varepsilon}{\partial\nn}=\frac{\partial \mu^\varepsilon}{\partial\nn}=\frac{\partial \sigma^\varepsilon}{\partial\nn}=0,\quad &\text{on}\ \partial\Omega\times(0,T).\label{model}
\end{array}
\right.
\end{eqnarray}
Here $\Omega\subset\mathbb{R}^N$ is a bounded smooth domain, $\nn$ is the unit outer normal to $\partial\Omega$, $\varepsilon^2$ is the diffusivity corresponding to the surface energy, $u^\varepsilon$ is the tumor cell concentration, $\mu^\varepsilon$ is  the chemical potential, $\sigma^\varepsilon$ is  the nutrient concentration,  and $f$ is a double equal-well potential taking its
global minimum value $0$ at $u=\pm1$. Without loss of generality  we take  $f(u)=(u^2-1)^2$.

The morphological evolution of tumor progression has been an
area of intense research interest recently(see, for instance \cite{BLM, CGRS, FBM, FGR, HVO, JWZ, LFJCLMWC, LTZ} and the references therein).
System (\ref{model}) is introduced to study the evolution of a growing solid tumor  which coexists with the host tissue. The dynamics can be divided into two stages. During the first stage,  two species are segregated according to the initial
data and interface appears around the common boundary of two species.  After a very fast time the dynamics enters the second stage in which the interface begins to involve. The second stage takes a much longer time than the first stage.

 Generation of the interface will be left in the forthcoming paper. In this paper we are interested in the later stage and assume that the interface
has been formed initially. There are two well-known approaches
to describe the motion of the interface so far.  The classical modeling approach is the so-called sharp interface
approach which treats the interface  between two phases as a $N-1$ dimensional sufficiently smooth
surface with zero width.  The second modeling
approach (the so-called diffuse interface approach)  treats the tumor/host tissue interface as a transition layer
with finite (small) width. Comparing with  the sharp interface model, the diffuse interface model has many advantages in numerical simulations of the
interfacial motion(see, for instance, \cite{CLLW} and the references therein). (\ref{model}) is a diffuse interface model related to the dynamics of tumor growth which consists of advection-reaction-diffusion equation coupled with the Cahn-Hilliard equation.

One of important and natural problems is to investigate whether the diffuse interface model
can be related to the corresponding sharp interface model by the asymptotic limit (i.e., the sharp interface limit) when the interfacial width tends to zero.  Some formal asymptotic analyses regarding the sharp interface limits of some different tumor-growth models can be found  in  \cite{GLSS, GLNS, HKNZ} for instance,  but, up to
our knowledge, only a few rigorous results  are set up  for such coupled systems.
The authors in \cite{RS} rewrote (\ref{model}) as a gradient flow   and used the techniques related to  gradient flow
 to prove that  (\ref{model}) converges to  the corresponding sharp interface model in the sense of  $\Gamma$-convergence as $\varepsilon\rightarrow 0 $.
One can see \cite{DFRSS} for  some rigorous sharp
interface limit for a model which is introduced in \cite{CWSL} and coupled with   the velocity field in a simplified case. More recently, the authors in \cite{RM} consider Cahn-Hilliard-Darcy system (first neglecting the nutrient $\sigma^\varepsilon$) that models tumor growth and prove that weak solutions  tend to varifold
solutions of a corresponding sharp interface model when the interface thickness goes to zero. One can see  \cite{AL1, CD, L, SS} for instance  for more works on  the convergence in the sense of $\Gamma$-convergence or varifold solutions, and \cite{AL2} for the sharp interface limit of the Stokes-Allen-Cahn system.

This paper focuses on the rigorous analysis on the sharp interface limit   of local classical solutions to (\ref{model}).  In \cite{ABC}  the authors
proved the classical solutions of the Cahn-Hilliard equation tend to solutions of the Mullins-Sekerka
problem (also called the Hele-Shaw problem) assuming the classical solutions of the latter exists. By employing the method used in \cite{ABC} we will  show more explicitly the asymptotic behaviors of local classical solutions $(u^\varepsilon,\sigma^\varepsilon,\mu^\varepsilon)$  in the sense of pointwise when $\varepsilon$ goes to zero and  establish  a  stronger convergence  than  the convergences  in the sense of $\Gamma$-convergence and varifold solutions in some sense.  In particular we can characterize  the evolution laws of $(u^\varepsilon,\sigma^\varepsilon,\mu^\varepsilon)$ in the transition region,  a small neighbour of $\Gamma$,  in which the behaviors are different from the ones in the two phase spaces.

The sharp interface model of  (\ref{model}) is the following two-phase flow(Theorem 5.8 and Theorem 5.9 in \cite{RS}):
\begin{eqnarray}
\left \{
\begin {array}{ll}
-\Delta \mu+\mu=2\sigma\pm1,\ \qquad\qquad\qquad&\text{in}\ \ \Omega_{\pm},\\[3pt]
\partial_t\sigma-\Delta \sigma+2\sigma=\mu\mp1,\ &\text{in}\ \ \Omega_{\pm},\\[3pt]
\mathbb{[\mu]}=\mathbb{[\sigma]}=0,\ &\text{on}\ \ \Gamma,\\[3pt]
\mathbb{[\frac{\partial\mu}{\partial\nn}]}=-2V,&\text{on}\ \ \Gamma,\\[3pt]
\mathbb{[\frac{\partial\sigma}{\partial\nn}]}=0,&\text{on}\ \ \Gamma,\\[3pt]
\mu=\kappa\int_{-1}^{1}\sqrt{2f(u)}du,&\text{on}\ \ \Gamma,\\[3pt]
\frac{\partial \mu}{\partial\nn}=\frac{\partial \sigma}{\partial\nn}=0,\quad &\text{on}\ \partial\Omega,\label{sharp interface model}
\end{array}
\right.
\end{eqnarray}
here $\Gamma$ is a closed sharp interface, $ \Omega_{-}$ and $\Omega_{+}$ are the interior and exterior of $\Gamma$ in $\Omega$ respectively, $\nn$ is the unit outer normal to $\Gamma$ from $ \Omega_{-}$ to $\Omega_{+}$ or to $\partial\Omega$, $V$ is the normal velocity of the sharp interface $\Gamma$, $\kappa$ is the mean curvature of  $\Gamma$ and $[f]$  denotes the jump condition of $f$ from $\Omega_+$ to $\Omega_-$ defined by $[f]=f|_{\Omega_+}-f|_{\Omega_-}$.

 Specifically speaking, 
%
firstly, we use Hilbert expansion method   to  construct an explicit approximate solution to (\ref{model}) around the local classical solution of (\ref{sharp interface model}) by assuming the latter exists and this process also can recover the sharp interface model (\ref{sharp interface model}). This method has been used in \cite{ABC, AKO, CCO, F, GN, WZZ2} and the references therein. In this paper Hilbert expansions will be performed in two phase space  $ \Omega_{\pm}$, transition layer and near boundary $\partial\Omega$. A kind of inner-outer matching condition will be imposed to ensure the outer expansions in $ \Omega_{\pm}$ and inner expansions in transition layer should match in the overlapped region. The same to the case that outer expansions in $ \Omega_{\pm}$ and boundary layer expansion near boundary $\partial\Omega$. Such a method is also called  matched asymptotic expansion method.

Secondly, based on the energy method, we derive the smallness of the error between the
approximate solution and the true solution. Consequently we can prove rigorously that (\ref{model}) converges to (\ref{sharp interface model}) as $\varepsilon\rightarrow 0.$
To estimate the error we mainly need to prove the following inequality
\begin{align*}
\int_{\Omega}\bigg(|\nabla v|^2+\frac{1}{\varepsilon^2}f''(u^A)v^2\bigg)dx\geq -C\int_{\Omega}v^2dx.
\end{align*}
holds for any small $\varepsilon$ and $v\in H^1(\Omega)$. Here $u^A$ is an approximate solution of $u^\varepsilon$ to be constructed and satisfies (\ref{profile of app}).  This inequality has been proved in \cite{C} and  used to prove the convergence of the Cahn-Hilliard equation to the Hele-Shaw model in \cite{ABC}. Here we write the details in
a concise way which shows clearly how to vanish the singularity $\frac{1}{\varepsilon^2}$. A sketch of the proof is as follows.

Noting that $f''(\pm1)>0$ and
\begin{align*}
u^A\sim\pm1, \ \text{in}\ \Omega_{\pm},
\end{align*}
then $\frac{1}{\varepsilon^2}f''(u^A)v^2$ is non-negative in $\Omega_{\pm}$ and thus we only need to controll $\frac{1}{\varepsilon^2}f''(u^A)v^2$ in the transition layer $\Gamma(\delta)$ which is defined in (\ref{gamma-neighbour}). According to the construction of the approximate solution,
\begin{align}
u^A\sim\theta(z)+\varepsilon\widetilde{u}^{(1)}(x,t,z)+O(\varepsilon^2), \ \ z\sim\frac{d^{(0)}(x,t)}{\varepsilon}+d^{(1)}(x,t), \ \text{in}\ \Gamma(\delta),\label{profile of app}
\end{align}
where  $\theta(z)$ is the solution to
\begin{align}
\theta''(z)=f'(\theta(z)),\ \ \theta(\pm\infty)=\pm1,\ \ \ \theta(0)=0,\label{ODE-0}
\end{align}
and $d^{(0)},d^{(1)}$ are defined in (\ref{d}), and $\widetilde{u}^{(1)}(x,t,z)$  which is defined in (\ref{in-1})  satisfies the important property Lemma \ref{u1:rewrite}.
Therefore one has
\begin{align*}
\frac{1}{\varepsilon^2}f''(u^A)v^2\sim \frac{1}{\varepsilon^2}f''(\theta(z))v^2+\frac{1}{\varepsilon} f'''(\theta(z))\widetilde{u}^{(1)}(x,t,z)v^2+O(1)v^2, \ \text{in}\ \Gamma(\delta).
\end{align*}

To deal with the most singular term $\frac{1}{\varepsilon^2}f''(\theta(z))v^2$ we will draw support from the diffusion term and  use the  estimates of the first eigenvalue, the corresponding eigenfunction and the second eigenvalue of the following  Neumann eigenvalue problem
\begin{align}
\mathcal{L}_fq:=-\frac{d^2q}{dz^2}+f''(\theta)q=\lambda q, \ \ \ z\in I_{\ve}=(-\frac{1}{\varepsilon},\frac{1}{\varepsilon});\ \ \ q'(\pm\frac{1}{\ve})=0,\label{operator:f}
\end{align}
which has been proved in \cite{C}(proved by a new method in \cite{FWZZ}). And the next singular term $\frac{1}{\varepsilon}\int_\Omega f'''(\theta(z))\widetilde{u}^{(1)}(x,t,z)v^2dx$ will vanish in some sense due to the property Lemma \ref{u1:rewrite} of $\widetilde{u}^{(1)}$.

To deal with the difficulty coming from the coupling term $2\sigma^\varepsilon+u^\varepsilon-\mu^\varepsilon$ in (\ref{model}), we consider the auxiliary variable $\varphi^\varepsilon=u^\varepsilon+\sigma^\varepsilon$ and construct an approximate solution $(\varphi^A,\sigma^A)$ to approximate the system which $(\varphi^\varepsilon,\sigma^\varepsilon)$ satisfies. The system for $(\varphi^\varepsilon,\sigma^\varepsilon)$ is a gradient-flow system(\cite{RS}).
Then we apply the energy method to estimate the errors $\varphi^\varepsilon-\varphi^A$ and $\sigma^\varepsilon-\sigma^A$.

Our main conclusion is
\begin{Theorem} \label{main theorem}  Given a  classical  solution $(\mu,\sigma,\Gamma)$ of (\ref{sharp interface model})  in $\Omega\times(0,T)$ which satisfies
\begin{align}
\text{dist}(\Gamma,\partial\Omega)>0,\ \ x\in\Omega,\ t\in[0,T].\label{positive distance}
\end{align}
Then there exist $\varepsilon_0>0$ and $C>0$ such that the following hold:
for any $0<\varepsilon<\varepsilon_0$, there exist $u^\varepsilon_0(x)$, $\sigma^\varepsilon_0(x)$ such that if $(u^\varepsilon, \sigma^\varepsilon)$ is the solution to (\ref{model}) in $\Omega\times(0,T)$, we have
\begin{align*}
\big\|u^\varepsilon-u^A\big\|_{C^{4,1}(\overline{\Omega}\times[0.T])}+\big\|\sigma^\varepsilon-\sigma^A\big\|_{C^{2,1}(\overline{\Omega}\times[0.T])}+\big\|\mu^\varepsilon-\mu^A\big\|_{C^{2,1}(\overline{\Omega}\times[0.T])}\leq C\varepsilon.
\end{align*}
where $(u^A, \sigma^A, \mu^A)$ which will be constructed in Section \ref{onstruction of an approximate solution} is an  approximate solution of  $(u^\varepsilon, \sigma^\varepsilon, \mu^\varepsilon)$  with the order of asymptotic expansion $k$ large enough.
\end{Theorem}
Aa a corollary we have
\begin{Corollary}\label{main corollary}Given a  classical  solution $(\mu,\sigma,\Gamma)$ of (\ref{sharp interface model})  as in Theorem \ref{main theorem}. Then there exist $u^\varepsilon_0(x)$, $\sigma^\varepsilon_0(x)$ such that if $(u^\varepsilon, \sigma^\varepsilon)$ is the solution to (\ref{model}) in $\Omega\times(0,T)$, then as $\varepsilon\rightarrow 0$, there hold
\begin{align}
&\big\|u^\varepsilon-(\pm 1)\big\|_{C(\Omega_{\pm}\backslash\Gamma(\delta))}\rightarrow 0,\ \ \bigg\|u^\varepsilon-\theta\bigg(\frac{d^{(0)}}{\varepsilon}+d^{(1)}\bigg)\bigg\|_{C(\Gamma(\delta))}\rightarrow 0,\label{main corollary-1}
\\&\big\|\sigma^\varepsilon-\sigma\big\|_{C(\overline{\Omega}\times[0.T])}+\big\|\mu^\varepsilon-\mu\big\|_{C(\overline{\Omega}\times[0.T])}\rightarrow 0,\label{main corollary-2}
\end{align}
here and in what follows $\delta$ is a small positive constant satisfying  $\delta<\frac{1}{2}\text{dist}(\Gamma,\partial\Omega)$ for $(x,t)\in\Omega\times[0,T]$.
\end{Corollary}

\begin{Remark}
 $u^\varepsilon_0(x)$ and $\sigma^\varepsilon_0(x)$ are defined in (\ref{initial data-impose}). The initial data like this form is often called sharp interface
initial data since we assume the interface has been formed initially. One can refer to \cite{ABC, LPW} for this kind of data.
\end{Remark}

We organize this paper as follows. In Section 2, we do $\varepsilon^k(k=0,1)$-order asymptotic expansion  and get  all the zeroth order terms and some  1th order terms. In Section 3 we establish a spectral condition and  give a  smallness estimation of the error between the approximate solution and the true solution. Then we complete the proof of Theorem \ref{main theorem} and  Corollary \ref{main corollary}. In Section 4, we do $\varepsilon^k(k\geq 2)$-order asymptotic expansion  and get  all the remained  terms. In Section 5, an approximate solution is constructed.

 Through this paper $C$ denotes a generic positive constant independent of small $\varepsilon$.
%
\indent

\section{Matched asymptotic expansion}
\indent

Let $\Gamma^\varepsilon$ be a smooth surface centered in  the transition layer. For any $t\in[0,T]$ for fixed $T>0$, let $d^\varepsilon(x,t)$ be the signed distance from $x$ to $\Gamma^\varepsilon$. Then $d^\varepsilon$ is smooth and $|\nabla d^\varepsilon|=1$ in a neighborhood of  $\Gamma^\varepsilon$. We  assume
\begin{align}
d^\varepsilon(x,t)&=d^{(0)}(x,t)+\varepsilon d^{(1)}(x,t)+\varepsilon^2 d^{(2)}(x,t)+\cdots,\label{d}
\end{align}
where $ d^{(0)}$ is a signed distance to $\Gamma$ and  $d^{(i)}(i\geq 1)$ is to be determined later.

 Since
\begin{align*}
1=|\nabla d^{\varepsilon}|^2&=|\nabla d^{(0)}|^2+2\varepsilon\nabla d^{(0)}\cdot\nabla d^{(1)}+\sum\limits_{k=2}^{+\infty}\varepsilon^k\bigg(\sum\limits_{i=0}^{k}\nabla d^{(i)}\cdot\nabla d^{(k-i)}\bigg)\\&=1+2\varepsilon\nabla d^{(0)}\cdot\nabla d^{(1)}+\sum\limits_{k=2}^{+\infty}\varepsilon^k\bigg(\sum\limits_{i=0}^{k}\nabla d^{(i)}\cdot\nabla d^{(k-i)}\bigg),
\end{align*}
then
\begin{eqnarray}
\nabla d^{(0)}\cdot\nabla d^{(k)}=\mathcal{D}_{k-1}(k\geq1)\triangleq\left \{
\begin {array}{ll}
0,\quad&k=1,\\
\\
-\frac{1}{2}\sum\limits_{i=1}^{k-1}\nabla d^{(i)}\cdot\nabla d^{(k-i)},\quad&k\geq2.\label{equ:gradient d}
\end{array}
\right.
\end{eqnarray}

Furthermore, we have $\Gamma=\{(x,t):d^{(0)}(x,t)=0\}$, $\Omega_{\pm}=\{(x,t):d^{(0)}(x,t)\gtrless 0\}$, $V=-d^{(0)}_t$ and $\kappa=-\Delta d^{(0)}$.
Defining
\begin{align}
\Gamma(\delta)=\{(x,t)\in\Omega\times(0,T):|d^{(0)}(x,t)|<\delta\}.\label{gamma-neighbour}
\end{align}

Now we  do outer expansion in $\Omega_{\pm}$, inner expansion in $\Gamma(\delta)$ and  boundary layer expansion in $\partial\Omega(\delta)=\{(x,t): \text{dist}(x,\partial\Omega)<\delta, x\in\Omega,t\in[0,T]\}$. For clarity  we only match zero-order and $\varepsilon$-order, and solve all the zero-order terms and some  1-order terms in this section. Matching $\varepsilon^k(k\geq2)$-order and all the remained  terms  will be presented in Section \ref{Matching k-order expansions}.
\subsection{Outer expansion in $\Omega_{\pm}$}
\indent

In $\Omega_{\pm}$, we  set
\begin{align}
u^\varepsilon=u^{(0)}_{\pm}+\varepsilon u^{(1)}_{\pm}+\varepsilon^2 u^{(2)}_{\pm}+\cdots,\label{out-1}\\
\mu^\varepsilon=\mu^{(0)}_{\pm}+\varepsilon \mu^{(1)}_{\pm}+\varepsilon^2 \mu^{(2)}_{\pm}+\cdots,\label{out-2}\\
\sigma^\varepsilon=\sigma^{(0)}_{\pm}+\varepsilon \sigma^{(1)}_{\pm}+\varepsilon^2 \sigma^{(2)}_{\pm}+\cdots.\label{out-3}
\end{align}

Moreover, by using Taylor expansion we  write  in $\Omega_{\pm}$
\begin{align*}
f'(u^\varepsilon)=f'(u^{(0)}_{\pm})+\varepsilon f''(u^{(0)}_{\pm})u^{(1)}_{\pm}+\cdots+\varepsilon^k\Big(f''(u^{(0)}_{\pm})u^{(k)}_{\pm}+g\big(u^{(0)}_{\pm},\cdots,u^{(k-1)}_{\pm}\big)\Big) +\cdots,
\end{align*}
here $g\big(u^{(0)}_{\pm},\cdots,u^{(k-1)}_{\pm}\big)$ depends on $u^{(0)}_{\pm},\cdots,u^{(k-1)}_{\pm}.$

Substituting (\ref{out-1})-(\ref{out-3}) into (\ref{model}) and collecting all terms of  zero order  we have
\begin{align*}
\partial_tu^{(0)}_{\pm}-\Delta \mu^{(0)}_{\pm}&=2\sigma^{(0)}_{\pm}+u^{(0)}_{\pm}-\mu^{(0)}_{\pm},\\
\partial_t\sigma^{(0)}_{\pm}-\Delta \sigma^{(0)}_{\pm}&=-\big(2\sigma^{(0)}_{\pm}+u^{(0)}_{\pm}-\mu^{(0)}_{\pm}\big),\\
f'(u^{(0)}_{\pm})&=0.
\end{align*}

We  take
\begin{align}
u_{\pm}^{(0)}=\pm1\label{out-leading}
\end{align}
and then
\begin{align}
-\Delta \mu^{(0)}_{\pm}+\mu^{(0)}_{\pm}&=2\sigma^{(0)}_{\pm}\pm1, \label{out-leading-1}  \\
\partial_t\sigma^{(0)}_{\pm}-\Delta \sigma^{(0)}_{\pm}+2\sigma^{(0)}_{\pm}&=\mu^{(0)}_{\pm}\mp1.\label{out-leading-2}
\end{align}

\indent
%
Substituting (\ref{out-1})-(\ref{out-3}) into (\ref{model}) and collecting all terms of  $\varepsilon$-order  we have
%
\begin{align*}
u^{(1)}_{\pm}&=\frac{\mu^{(0)}_{\pm}}{f''(u^{(0)}_{\pm})},\\
-\Delta \mu^{(1)}_{\pm}+\mu^{(1)}_{\pm}&=2\sigma^{(1)}_{\pm}+u^{(1)}_{\pm}-\partial_tu^{(1)}_{\pm},\\
\partial_t\sigma^{(1)}_{\pm}-\Delta \sigma^{(1)}_{\pm}+2\sigma^{(1)}_{\pm}&=\mu^{(1)}_{\pm}-u^{(1)}_{\pm}.
\end{align*}
%
\subsection{Inner expansion in $\Gamma(\delta)$}
\indent

Let  $z=\frac{d^\varepsilon}{\varepsilon}\in(-\infty,+\infty)$. In $\Gamma(\delta)$, we  set
\begin{align}
u^\varepsilon(x,t)=\widetilde{u}^\varepsilon(x,t,z)&=\widetilde{u}^{(0)}(x,t,z)+\varepsilon \widetilde{u}^{(1)}(x,t,z)+\varepsilon^2 \widetilde{u}^{(2)}(x,t,z)+\cdots,\label{in-1}\\
\mu^\varepsilon(x,t)=\widetilde{\mu}^\varepsilon(x,t,z)&=\widetilde{\mu}^{(0)}(x,t,z)+\varepsilon \widetilde{\mu}^{(1)}(x,t,z)+\varepsilon^2 \widetilde{\mu}^{(2)}(x,t,z)+\cdots,\label{in-2}\\
\sigma^\varepsilon(x,t)=\widetilde{\sigma}^\varepsilon(x,t,z)&=\widetilde{\sigma}^{(0)}(x,t,z)+\varepsilon \widetilde{\sigma}^{(1)}(x,t,z)+\varepsilon^2 \widetilde{\sigma}^{(2)}(x,t,z)+\cdots,\label{in-3}
\end{align}
and the following inner-outer matching conditions: there exists a fixed positive constant $\nu$  such  that as $z\rightarrow\pm\infty$   there hold
\begin{align}
&D_x^\alpha D_t^\beta D_z^\gamma\Big(\widetilde{u}^{(i)}(x,t,z)- u^{(i)}_{\pm}(x,t)\Big)=O(e^{-\nu|z|}),\label{mc-1}\\
&D_x^\alpha D_t^\beta D_z^\gamma\Big(\widetilde{\mu}^{(i)}(x,t,z) -\mu^{(i)}_{\pm}(x,t)\Big)=O(e^{-\nu|z|}),\label{mc-2}\\
&D_x^\alpha D_t^\beta D_z^\gamma\Big(\widetilde{\sigma}^{(i)}(x,t,z)- \sigma^{(i)}_{\pm}(x,t)\Big)=O(e^{-\nu|z|})\label{mc-3}
\end{align}
for $(x,t)\in\Gamma(\delta)$ and $0\leq\alpha,\beta,\gamma\leq N$ with $N$ depending on actual expansion order.

Using Taylor expansion and (\ref{in-1}) one has
\begin{align}
f'(\widetilde{u}^\varepsilon)=f'(\widetilde{u}^{(0)})+\varepsilon f''(\widetilde{u}^{(0)})\widetilde{u}^{(1)}+\cdots+\varepsilon^k\Big(f''(\widetilde{u}^{(0)})\widetilde{u}^{(k)}
+\widetilde{g}\big(\widetilde{u}^{(0)},\cdots,\widetilde{u}^{(k-1)}\big)\Big)+\cdots,\label{inequ-4}
\end{align}
here $g\big(\widetilde{u}^{(0)},\cdots,\widetilde{u}^{(k-1)}\big)$ depends on $\widetilde{u}^{(0)},\cdots,\widetilde{u}^{(k-1)}.$

Noting that
\begin{align*}
\widetilde{u}_t^\varepsilon(x,t,z)&=\partial_t\widetilde{u}^\varepsilon+\varepsilon^{-1}\partial_z\widetilde{u}^\varepsilon\partial_td^\varepsilon,\\
\Delta_x\widetilde{\mu}^\varepsilon(x,t,z)&=\varepsilon^{-2}\partial_{zz}\widetilde{\mu}^\varepsilon+2\varepsilon^{-1}\nabla_x\partial_z\widetilde{\mu}^\varepsilon\cdot\nabla_x d^\varepsilon+\varepsilon^{-1}\partial_z\widetilde{\mu}^\varepsilon\Delta_x d^\varepsilon+\Delta_x\widetilde{\mu}^\varepsilon,
\end{align*}
then in the new coordinate $(x,t,z)$ the first equation in (\ref{model}) becomes
\begin{align}
-\partial_{zz}\widetilde{\mu}^\varepsilon+&\varepsilon\Big(\partial_z\widetilde{u}^\varepsilon\partial_td^\varepsilon
-2\nabla_x\partial_z\widetilde{\mu}^\varepsilon\cdot\nabla_x d^\varepsilon-\partial_z\widetilde{\mu}^\varepsilon\Delta_x d^\varepsilon \Big)\nonumber
\\+&\varepsilon^2\Big(\partial_t\widetilde{u}^\varepsilon-\Delta_x \widetilde{\mu}^\varepsilon-\big(2\widetilde{\sigma}^\varepsilon+\widetilde{u}^\varepsilon-\widetilde{\mu}^\varepsilon\big)\Big)=0.\label{inequ-1}
\end{align}

Similarly, the second equation and the third equation in (\ref{model}) become respectively
\begin{align}
&-\partial_{zz}\widetilde{\sigma}^\varepsilon+\varepsilon\Big(\partial_z\widetilde{\sigma}^\varepsilon\partial_td^\varepsilon
-2\nabla_x\partial_z\widetilde{\sigma}^\varepsilon\cdot\nabla_x d^\varepsilon-\partial_z\widetilde{\sigma}^\varepsilon\Delta_x d^\varepsilon \Big)\nonumber\\ &\quad\ \ \qquad+\varepsilon^2\Big(\partial_t\widetilde{\sigma}^\varepsilon-\Delta_x \widetilde{\sigma}^\varepsilon+2\widetilde{\sigma}^\varepsilon+\widetilde{u}^\varepsilon-\widetilde{\mu}^\varepsilon\Big)=0,\label{inequ-2}\\
&-\partial_{zz}\widetilde{u}^\varepsilon+f'(\widetilde{u}^\varepsilon)-\varepsilon\Big(2\nabla_x\partial_z\widetilde{u}^\varepsilon\cdot\nabla_x d^\varepsilon+\partial_z\widetilde{u}^\varepsilon\Delta_x d^\varepsilon+\widetilde{\mu}^\varepsilon\Big)-\varepsilon^2\Delta_x\widetilde{u}^\varepsilon=0.\label{inequ-3}
\end{align}

Generally the  inner-outer exponentially decaying matching conditions (\ref{mc-1})-(\ref{mc-3}) may not necessarily hold. To ensure   this   conditions,  we will modify (\ref{inequ-1})-(\ref{inequ-3}) as follows motivated by \cite{ABC}.  For that we choose a smooth non-decreasing  function $\eta$ such that $\eta(z)=0$ for $z\leq -1$, $\eta(z)=1$ for $z\geq 1$ and define
\begin{align*}
\eta^{\pm}(z)=\eta(-M\pm z),\ \ \ z\in \mathbb{R},
\end{align*}
here the constant $M=\|d^{(1)}\|_{C^0(\Gamma(\delta))}+2$.

Now we modify (\ref{inequ-1})-(\ref{inequ-3}) as follows
 \begin{align}
&-\partial_{zz}\widetilde{\mu}^\varepsilon+\varepsilon\Big(\partial_z\widetilde{u}^\varepsilon\partial_td^\varepsilon
-2\nabla_x\partial_z\widetilde{\mu}^\varepsilon\cdot\nabla_x d^\varepsilon-\partial_z\widetilde{\mu}^\varepsilon\Delta_x d^\varepsilon \Big)\nonumber
\\&\quad\ \ \qquad+\varepsilon^2\Big(\partial_t\widetilde{u}^\varepsilon-\Delta_x \widetilde{\mu}^\varepsilon-\big(2\widetilde{\sigma}^\varepsilon+\widetilde{u}^\varepsilon-\widetilde{\mu}^\varepsilon\big)\Big)
\nonumber\\&\quad\ \ \qquad+\eta'' p^\varepsilon(d^\varepsilon-\varepsilon z)+\varepsilon \eta' g^\varepsilon(d^\varepsilon-\varepsilon z)-\varepsilon^2(s_{+}^\varepsilon\eta^++s_{-}^\varepsilon \eta^-)=0,\label{inequ-m-1}\\
&-\partial_{zz}\widetilde{\sigma}^\varepsilon+\varepsilon\Big(\partial_z\widetilde{\sigma}^\varepsilon\partial_td^\varepsilon
-2\nabla_x\partial_z\widetilde{\sigma}^\varepsilon\cdot\nabla_x d^\varepsilon-\partial_z\widetilde{\sigma}^\varepsilon\Delta_x d^\varepsilon \Big)\nonumber\\ &\quad\ \ \qquad+\varepsilon^2\Big(\partial_t\widetilde{\sigma}^\varepsilon-\Delta_x \widetilde{\sigma}^\varepsilon+2\widetilde{\sigma}^\varepsilon+\widetilde{u}^\varepsilon-\widetilde{\mu}^\varepsilon\Big)
\nonumber\\&\quad\ \ \qquad+\eta'' q^\varepsilon(d^\varepsilon-\varepsilon z)+\varepsilon \eta' h^\varepsilon(d^\varepsilon-\varepsilon z)-\varepsilon^2(r_{+}^\varepsilon\eta^++r_{-}^\varepsilon\eta^-)=0,\label{inequ-m-2}\\[5pt]
&-\partial_{zz}\widetilde{u}^\varepsilon+f'(\widetilde{u}^\varepsilon)-\varepsilon\big(2\nabla_x\partial_z\widetilde{u}^\varepsilon\cdot\nabla_x d^\varepsilon+\partial_z\widetilde{u}^\varepsilon\Delta_x d^\varepsilon+\widetilde{\mu}^\varepsilon\big)-\varepsilon^2\Delta_x\widetilde{u}^\varepsilon
\nonumber\\[5pt]
&\quad\ \ \qquad+\varepsilon\eta' l^\varepsilon(d^\varepsilon-\varepsilon z)=0,\label{inequ-m-3}
\end{align}
 where
 \begin{align*}
&g^\varepsilon(x,t)=\sum\limits_{i=0}^{+\infty}\varepsilon^ig^{(i)}(x,t),\ \ h^\varepsilon(x,t)=\sum\limits_{i=0}^{+\infty}\varepsilon^ih^{(i)}(x,t),\ \ l^\varepsilon(x,t)=\sum\limits_{i=0}^{+\infty}\varepsilon^il^{(i)}(x,t),
\\&p^\varepsilon(x,t)=\sum\limits_{i=0}^{+\infty}\varepsilon^ip^{(i)}(x,t),\ \ q^\varepsilon(x,t)=\sum\limits_{i=0}^{+\infty}\varepsilon^iq^{(i)}(x,t),
\\&s_{\pm}^\varepsilon(x,t)=\sum\limits_{i=0}^{+\infty}\varepsilon^is_{\pm}^{(i)}(x,t)=\sum\limits_{i=0}^{+\infty}\varepsilon^i\big(\partial_tu_\pm^{(i)}-\Delta \mu_\pm^{(i)}-\big(2\sigma_\pm^{(i)}+u_\pm^{(i)}-\mu_\pm^{(i)}\big)\big)(x,t),
\\&r_{\pm}^\varepsilon(x,t)=\sum\limits_{i=0}^{+\infty}\varepsilon^ir_{\pm}^{(i)}(x,t)=\sum\limits_{i=0}^{+\infty}\varepsilon^i\big(
\partial_t\sigma_\pm^{(i)}-\Delta \sigma_\pm^{(i)}-\big(2\sigma_\pm^{(i)}+u_\pm^{(i)}-\mu_\pm^{(i)}\big)\big)(x,t).
\end{align*}

\begin{Remark}
Due to $z=\frac{d^\varepsilon}{\varepsilon}$, then $d^\varepsilon-\varepsilon z=0$. Moreover, since $M=\|d^{(1)}\|_{C^0(\Gamma(\delta))}+2$, we can find $r_{+}^\varepsilon\eta^++r_{-}^\varepsilon\eta^-=0$ for $(x,t)\in\Gamma(\delta)$ and small $\varepsilon$. One can refer to Remark 4.2 in \cite{ABC} for the details. Therefore (\ref{inequ-1})-(\ref{inequ-3}) are the same to (\ref{inequ-m-1})-(\ref{inequ-m-3}) respectively. However, through the modifications, we have changed the equation of every order such that the matching conditions (\ref{mc-1})-(\ref{mc-3}) hold.
\end{Remark}

For clarity we divide into two subsections to proceed.
\subsubsection{Matching zeroth order}\label{Matching zero-order}
\indent

Substituting (\ref{d}) and (\ref{in-1})-(\ref{in-3}) into (\ref{inequ-m-1})-(\ref{inequ-m-3})  and collecting all terms of  zero order  we have
\begin{align}
&\partial_{zz}\widetilde{\mu}^{(0)}=\eta''p^{(0)}d^{(0)},\label{in-0-1}\\
&\partial_{zz}\widetilde{\sigma}^{(0)}=\eta''q^{(0)}d^{(0)},\label{in-0-2}\\
&\partial_{zz}\widetilde{u}^{(0)}=f'(\widetilde{u}^{(0)}).\label{in-0-3}
\end{align}

Now we argue  in this subsection according to the following order:
\begin{align*}
(\widetilde{\mu}^{(0)},\widetilde{\sigma}^{(0)})\rightarrow (p^{(0)},q^{(0)},\big[\mu^{(0)}\big],\big[\sigma^{(0)}\big])\rightarrow\widetilde{u}^{(0)}.
\end{align*}

$\bullet$ $(\widetilde{\mu}^{(0)},\widetilde{\sigma}^{(0)})$

From (\ref{in-0-1}) we can write
\begin{align*}
\widetilde{\mu}^{(0)}(z,x,t)=\eta(z)p^{(0)}(x,t)d^{(0)}(x,t)+a(x,t)z+b(x,t)
\end{align*}
for some $a(x,t)$ and $b(x,t)$. Since the inner-outer matching condition $\widetilde{\mu}^{(0)}(z,x,t)\rightarrow \mu_{\pm}^{(0)}(x,t)$ as $z\rightarrow \pm\infty$ needs to be satisfied, we must have
\begin{align*}
a(x,t)=0,\ b(x,t)=\mu_{-}^{(0)}(x,t)
\end{align*}
and
\begin{align}
p^{(0)}(x,t)d^{(0)}(x,t)=\mu_{+}^{(0)}(x,t)-\mu_{-}^{(0)}(x,t).\label{p0}
\end{align}

Thus we have
\begin{align}
\widetilde{\mu}^{(0)}(x,t,z)=\eta(z)\mu_{+}^{(0)}(x,t)+\big(1-\eta(z)\big)\mu_{-}^{(0)}(x,t)\label{definition:mu0-in}
\end{align}
and then
\begin{align*}
D_x^\alpha D_t^\beta D_z^\gamma\Big(\widetilde{\mu}^{(0)}(x,t,z) -\mu^{(0)}_{\pm}(x,t)\Big)=O(e^{-\nu|z|})
\end{align*}
for any $\alpha,\beta,\gamma\in\mathbb{N}$ and $\nu>0$.

Similarly, we get
\begin{align}
q^{(0)}(x,t)d^{(0)}(x,t)=\sigma_{+}^{(0)}(x,t)-\sigma_{-}^{(0)}(x,t)\label{q0}
\end{align}
and
\begin{align}
\widetilde{\sigma}^{(0)}(x,t,z)=\eta(z)\sigma_{+}^{(0)}(x,t)+\big(1-\eta(z)\big)\sigma_{-}^{(0)}(x,t)\label{definition:sigma0-in}
\end{align}
which implies
\begin{align*}
D_x^\alpha D_t^\beta D_z^\gamma\Big(\widetilde{\sigma}^{(0)}(x,t,z) -\sigma^{(0)}_{\pm}(x,t)\Big)=O(e^{-\nu|z|})
\end{align*}
for any $\alpha,\beta,\gamma\in\mathbb{N}$ and $\nu>0$.

$\bullet$ $(p^{(0)},q^{(0)},\big[\mu^{(0)}\big],\big[\sigma^{(0)}\big])$

Moreover, according to (\ref{p0}) and (\ref{q0}) there hold on $\Gamma$
\begin{align}
\big[\mu^{(0)}\big]\triangleq\mu_{+}^{(0)}-\mu_{-}^{(0)}=0,\ \ \big[\sigma^{(0)}\big]\triangleq\sigma_{+}^{(0)}-\sigma_{-}^{(0)}=0.\label{jump condition:mu0}
\end{align}
And we can define  smooth functions $p^{(0)}$ and $q^{(0)}$ in $\Gamma(\delta)$ as follows
\begin{eqnarray}
p^{(0)}=\left \{
\begin {array}{ll}
\frac{\mu_{+}^{(0)}-\mu_{-}^{(0)}}{d^{(0)}},\quad&\text{in}\ \ \Gamma(\delta)\backslash\Gamma,\\
\\
\nabla_x d^{(0)}\cdot\nabla_x(\mu_{+}^{(0)}-\mu_{-}^{(0)}),\quad &\text{on} \ \ \Gamma,
\end{array}\label{definition:p0}
\right.
\end{eqnarray}
and
\begin{eqnarray}
q^{(0)}=\left \{
\begin {array}{ll}
\frac{\sigma_{+}^{(0)}-\sigma_{-}^{(0)}}{d^{(0)}},\quad&\text{in}\ \ \Gamma(\delta)\backslash\Gamma,\\
\\
\nabla_x d^{(0)}\cdot\nabla_x(\sigma_{+}^{(0)}-\sigma_{-}^{(0)}),\quad &\text{on} \ \ \Gamma.
\end{array}\label{definition:q0}
\right.
\end{eqnarray}

$\bullet$ $\widetilde{u}^{(0)}$

By  (\ref{in-0-3}) and the inner-outer matching condition, $\widetilde{u}^{(0)}$ satisfies
\begin{align*}
\partial_{zz}\widetilde{u}^{(0)}=f'(\widetilde{u}^{(0)}),\ \ \widetilde{u}^{(0)}(\pm\infty)=u_{\pm}^{(0)}=\pm1,\ \ \ \widetilde{u}^{(0)}(0)=0,
\end{align*}
here   the condition $\widetilde{u}^{(0)}(0)=0$ is imposed to  ensure  $\widetilde{u}^{(0)}$ is unique. Therefore $\widetilde{u}^{(0)}$ is independent of $(x,t)$ and then $\widetilde{u}^{(0)}(x,t,z)=\theta(z)$ which is defined in (\ref{ODE-0}). In fact, we can get
\begin{align}
\widetilde{u}^{(0)}(x,t,z)=\theta(z)=\tanh(\sqrt{2}z)\label{in-leading}
 \end{align}
and for $k\in \mathbb{N}\cup\{0\}$
\begin{align*}
\frac{d^k}{dz^k}\big(\theta(z)+1\big)=O(e^{-\sqrt{2}|z|}), \ \text{as }z\to-\infty;\ \frac{d^k}{dz^k}\big(\theta(z)-1\big)=O(e^{-\sqrt{2}|z|}),\ \text{as }z\to+\infty,
\end{align*}
which implies
\begin{align*}
D_x^\alpha D_t^\beta D_z^\gamma\Big(\widetilde{u}^{(0)}(x,t,z) -u^{(0)}_{\pm}(x,t)\Big)=O(e^{-\sqrt{2}|z|})
\end{align*}
for any $\alpha,\beta,\gamma\in\mathbb{N}$.

\subsubsection{Matching 1th order}
\indent

Substituting (\ref{d}) and (\ref{in-1})-(\ref{in-3}) into (\ref{inequ-m-1})-(\ref{inequ-m-3})  and collecting all terms of  $\varepsilon$-order  we have
 \begin{align}
&-\partial_{zz}\widetilde{\mu}^{(1)}+\Big(\partial_z\widetilde{u}^{(0)}\partial_td^{(0)}-2\nabla_x\partial_z\widetilde{\mu}^{(0)}\cdot\nabla_x d^{(0)}-\partial_z\widetilde{\mu}^{(0)}\Delta_x d^{(0)} \Big)
\nonumber\\&\quad\ \ \ \ \qquad+\eta'' \big(p^{(1)}d^{(0)}+p^{(0)}d^{(1)}\big)-\eta''zp^{(0)}+ \eta' g^{(0)}d^{(0)}=0,\label{in-1-1}\\
&-\partial_{zz}\widetilde{\sigma}^{(1)}+\Big(\partial_z\widetilde{\sigma}^{(0)}\partial_td^{(0)}-2\nabla_x\partial_z\widetilde{\sigma}^{(0)}\cdot\nabla_x d^{(0)}-\partial_z\widetilde{\sigma}^{(0)}\Delta_x d^{(0)} \Big)
\nonumber\\&\quad\ \ \ \ \qquad+\eta'' \big(q^{(1)}d^{(0)}+q^{(0)}d^{(1)}\big)-\eta''zq^{(0)}+ \eta'h^{(0)}d^{(0)}=0,\label{in-1-2}\\
&-\partial_{zz}\widetilde{u}^{(1)}+f''(\widetilde{u}^{(0)})\widetilde{u}^{(1)}-\Big(2\nabla_x\partial_z\widetilde{u}^{(0)}\cdot\nabla_x d^{(0)}+\partial_z\widetilde{u}^{(0)}\Delta_x d^{(0)}+\widetilde{\mu}^{(0)}\Big)
\nonumber\\&\quad\ \ \ \ \qquad+\eta' l^{(0)}d^{(0)}=0.\label{in-1-3}
\end{align}

Next we argue  according to the following order:
\begin{align*}
\bigg(\widetilde{\mu}^{(1)},g^{(0)},\Big[\frac{\partial\mu^{(0)}}{\partial\nn}\Big]\bigg)\rightarrow\bigg(\widetilde{\sigma}^{(1)},h^{(0)},
\Big[\frac{\partial\sigma^{(0)}}{\partial\nn}\Big]\bigg)\rightarrow\bigg(\widetilde{u}^{(1)},l^{(0)},\mu_{\pm}^{(0)}\big|_{\Gamma}\bigg).
\end{align*}

$\bullet$ $\Big(\widetilde{\mu}^{(1)},g^{(0)},\Big[\frac{\partial\mu^{(0)}}{\partial\nn}\Big]\Big)$

For $(x,t)\in\Gamma(\delta)$, we write (\ref{in-1-1}) as
\begin{align}
-\Big(\widetilde{\mu}^{(1)}- \eta\big(p^{(1)}d^{(0)}+p^{(0)}d^{(1)}\big)\Big)_{zz}&=\eta''zp^{(0)}- \eta' g^{(0)}d^{(0)}-\partial_z\widetilde{u}^{(0)}\partial_td^{(0)}
\nonumber\\&\quad+2\nabla_x\partial_z\widetilde{\mu}^{(0)}\cdot\nabla_x d^{(0)}+\partial_z\widetilde{\mu}^{(0)}\Delta_x d^{(0)}
\nonumber\\&\triangleq \Theta_{0,1}.\label{in-1-1-0}
\end{align}

It follows from Lemma 4.3 in \cite{ABC} and direct computations that if
\begin{align}
\int_{-\infty}^{+\infty}\Theta_{0,1}(x,t,z)dz=0,\label{in-1-solving condition-1}
\end{align}
then (\ref{in-1-1-0}) has a bounded solution
\begin{align*}
\widetilde{\mu}^{(1)}(x,t,z)&=\eta(z)\big(p^{(1)}d^{(0)}+p^{(0)}d^{(1)}\big)(x,t)
\\&\quad+\int_{-\infty}^{z}\int_{z'}^{+\infty}\Theta_{0,1}(z'',x,t)dz''dz'+\mu_-^{(1)}(x,t)
\end{align*}
which and $\widetilde{\mu}^{(1)}(+\infty,x,t)=\mu_+^{(1)}(x,t)$ imply
\begin{align*}
\big(p^{(1)}d^{(0)}+p^{(0)}d^{(1)}\big)(x,t)=\mu_+^{(1)}(x,t)-\int_{-\infty}^{+\infty}\int_{z'}^{+\infty}\Theta_{0,1}(z'',x,t)dz''dz'-\mu_-^{(1)}(x,t).
\end{align*}

Hence we  obtain
\begin{align*}
\widetilde{\mu}^{(1)}(x,t,z)=&\eta(z)\mu_+^{(1)}(x,t)+\big(1-\eta(z)\big)\mu_-^{(1)}(x,t)\nonumber\\&
-\eta(z)\int_{-\infty}^{+\infty}\int_{z'}^{+\infty}\Theta_{0,1}(z'',x,t)dz''dz'
\nonumber\\&+\int_{-\infty}^{z}\int_{z'}^{+\infty}\Theta_{0,1}(z'',x,t)dz''dz'
\end{align*}
and
\begin{align*}
\big[\mu^{(1)}\big]\triangleq\mu_{+}^{(1)}-\mu_{-}^{(1)}=p^{(0)}d^{(1)}+\int_{-\infty}^{+\infty}\int_{z'}^{+\infty}\Theta_{0,1}(z'',x,t)dz''dz',\ \  \ \text{on}\ \ \Gamma.
\end{align*}

According to the results obtained in subsection \ref{Matching zero-order} one has for any $\alpha,\beta,\gamma\in\mathbb{N}$,
\begin{align*}
D_x^\alpha D_t^\beta D_z^\gamma\Theta_{0,1}=O(e^{-\nu|z|}),\ \ \text{for\ some }\ \nu>0,
\end{align*}
and thus
\begin{align*}
D_x^\alpha D_t^\beta D_z^\gamma\Big(\widetilde{\mu}^{(1)}(x,t,z) -\mu^{(1)}_{\pm}(x,t)\Big)=O(e^{-\nu|z|}),\ \ \text{for\ some }\ \nu>0.
\end{align*}

Moreover, by (\ref{in-1-solving condition-1}), (\ref{jump condition:mu0}), (\ref{definition:p0}) and direct computations we can get
\begin{align}
\bigg[\frac{\partial\mu^{(0)}}{\partial\nn}\bigg]\triangleq\nabla_x\big(\mu_{+}^{(0)}-\mu_{-}^{(0)}\big)\cdot\nabla_x d^{(0)}= 2d_t^{(0)}\triangleq -2V,\ \ \text{on}\ \Gamma,\label{jump condition:mun0}
\end{align}
and define a smooth function $g^{(0)}$ in $\Gamma(\delta)$
\begin{eqnarray}
g^{(0)}=\left \{
\begin {array}{ll}
\frac{(\mu_{+}^{(0)}-\mu_{-}^{(0)})\triangle_x d^{(0)}+2\nabla_x(\mu_{+}^{(0)}-\mu_{-}^{(0)})\cdot\nabla_x d^{(0)}-p^0-2\partial_td^{(0)}}{d^{(0)}},\quad&\text{in}\ \ \Gamma(\delta)\backslash\Gamma,\\
\\
\nabla_x d^{(0)}\cdot\nabla_x\big((\mu_{+}^{(0)}-\mu_{-}^{(0)})\Delta_x d^{(0)}+2\nabla_x(\mu_{+}^{(0)}-\mu_{-}^{(0)})\cdot\nabla_x d^{(0)}
\\ \qquad\qquad\qquad\qquad-p^0-2\partial_td^{(0)}\big),\quad &\text{on} \ \ \Gamma.
\end{array}\label{definition:g0}
\right.
\end{eqnarray}

$\bullet$ $\Big(\widetilde{\sigma}^{(1)},h^{(0)},\big[\frac{\partial\sigma^{(0)}}{\partial\nn}\big]\Big)$

Applying the above similar arguments to (\ref{in-1-2}), we obtain  that if
\begin{align}
\int_{-\infty}^{+\infty}\Theta_{0,2}(x,t,z)dz=0,\label{in-1-solving condition-1-1}
\end{align}
then (\ref{in-1-2}) has a bounded solution
\begin{align*}
\widetilde{\sigma}^{(1)}(x,t,z)=&\eta(z)\sigma_+^{(1)}(x,t)+\big(1-\eta(z)\big)\sigma_-^{(1)}(x,t)\nonumber\\&
-\eta(z)\int_{-\infty}^{+\infty}\int_{z'}^{+\infty}\Theta_{0,2}(x,t,z'')dz''dz'
\nonumber\\&+\int_{-\infty}^{z}\int_{z'}^{+\infty}\Theta_{0,2}(x,t,z'')dz''dz',
\end{align*}
and for any $\alpha,\beta,\gamma\in\mathbb{N}$,
\begin{align*}
D_x^\alpha D_t^\beta D_z^\gamma\Big(\widetilde{\sigma}^{(1)}(x,t,z) -\sigma^{(1)}_{\pm}(x,t)\Big)=O(e^{-\nu|z|}),\ \ \text{for\ some }\ \nu>0,
\end{align*}
and
\begin{align*}
\big[\sigma^{(1)}\big]\triangleq\sigma_{+}^{(1)}-\sigma_{-}^{(1)}=q^{(0)}d^{(1)}+\int_{-\infty}^{+\infty}\int_{z'}^{+\infty}\Theta_{0,2}(x,t,z'')dz''dz',\ \  \ \text{on}\ \ \Gamma,
\end{align*}
where
\begin{align*}
\Theta_{0,2}\triangleq \eta''zq^{(0)}- \eta' h^{(0)}d^{(0)}-\partial_z\widetilde{\sigma}^{(0)}\partial_td^{(0)}
+2\nabla_x\partial_z\widetilde{\sigma}^{(0)}\cdot\nabla_x d^{(0)}+\partial_z\widetilde{\sigma}^{(0)}\triangle_x d^{(0)}.
\end{align*}
%
Moreover, it follows from (\ref{in-1-solving condition-1-1}), (\ref{jump condition:mu0}), (\ref{definition:q0}) and direct computations that
\begin{align}
\bigg[\frac{\partial\sigma^{(0)}}{\partial\nn}\bigg]\triangleq\nabla\big(\sigma_{+}^{(0)}-\sigma_{-}^{(0)}\big)\cdot\nabla d^{(0)}= 0,\ \ \text{on}\ \Gamma,\label{jump condition:sigman0}
\end{align}
and
\begin{eqnarray}
h^{(0)}=\left \{
\begin {array}{ll}
\frac{(\sigma_{+}^{(0)}-\sigma_{-}^{(0)})(\triangle_x d^{(0)}-\partial_td^{(0)})+2\nabla_x(\sigma_{+}^{(0)}-\sigma_{-}^{(0)})\cdot\nabla_x d^{(0)}-q^0}{d^{(0)}},\quad&\text{in}\ \ \Gamma(\delta)\backslash\Gamma,\\
\\
\nabla_x d^{(0)}\cdot\nabla_x\big((\sigma_{+}^{(0)}-\sigma_{-}^{(0)})(\Delta_x d^{(0)}-\partial_td^{(0)})
\\ \qquad\qquad\qquad+2\nabla_x(\sigma_{+}^{(0)}-\sigma_{-}^{(0)})\cdot\nabla_x d^{(0)}-q^0\big),\quad &\text{on} \ \ \Gamma.
\end{array}\label{definition:h0}
\right.
\end{eqnarray}

$\bullet$ $\big(\widetilde{u}^{(1)},l^{(0)},\mu_{\pm}^{(0)}\big|_{\Gamma}\big)$

Based on the method of variation of constants for ODE and direct computations (or Lemma 4.3 in \cite{ABC}) we find that if
 \begin{align}
\int_{-\infty}^{+\infty}\Theta_{0,3}(x,t,z)\theta'(z)dz=0,\label{in-1-solving condition-2}
\end{align}
then the solution to (\ref{in-1-3})  with $\widetilde{u}^{(1)}(0,x,t)=1$ is
\begin{align*}
\widetilde{u}^{(1)}(x,t,z)=\frac{\theta'(z)}{\theta'(0)}+\theta'(z)\int_{0}^{z}(\theta'(\varsigma))^{-2}\int_\varsigma^{+\infty}\Theta_{0,3}(x,t,\tau)\theta'(\tau)d\tau d\varsigma
\end{align*}
which satisfies for any $\alpha,\beta,\gamma\in\mathbb{N}$,
\begin{align*}
D_x^\alpha D_t^\beta D_z^\gamma\Big(\widetilde{u}^{(1)}(x,t,z) -u^{(1)}_{\pm}(x,t)\Big)=O(e^{-\nu|z|}),\ \ \text{for\ some }\ \nu>0,
\end{align*}
where
\begin{align*}
\Theta_{0,3}&\triangleq2\nabla_x\partial_z\widetilde{u}^{(0)}\cdot\nabla_x d^{(0)}+\partial_z\widetilde{u}^{(0)}\Delta_x d^{(0)}+\widetilde{\mu}^{(0)}-\eta' l^0d^{(0)}
\nonumber\\&=\theta'\Delta_x d^{(0)}+\widetilde{\mu}^{(0)}-\eta' l^0d^{(0)}.
\end{align*}

%
Due to (\ref{in-1-solving condition-2}) there holds
\begin{align}
\mu_{\pm}^{(0)}(x,t)&=-\Delta_x d^{(0)}\int_{-\infty}^{+\infty}(\theta'(z))^2dz=\kappa\int_{-\infty}^{+\infty}(\theta'(z))^2dz
\nonumber\\&=2\kappa\int_{-\infty}^{+\infty}f(\theta(z))dz=\kappa\int_{-1}^{1}\sqrt{2f(u)}du,\ \ \  \ \text{on}\  \ \Gamma,\label{curvature condition}
\end{align}
here we have used $\theta'(z)=\sqrt{2f(\theta(z))}$.

Furthermore, by (\ref{in-1-solving condition-2})  a smooth function $l^{(0)}$ is defined as follows:
\begin{eqnarray}
l^{(0)}=\left \{
\begin {array}{ll}
\frac{1}{d^{(0)}\int_{-\infty}^{+\infty}\eta'\theta'dz}\Big(\Delta_x d^{(0)}\int_{-\infty}^{+\infty}(\theta')^2dz+\int_{-\infty}^{+\infty}\widetilde{\mu}^{(0)}\theta'dz\Big),\quad&\text{in}\ \ \Gamma(\delta)\backslash\Gamma,\\
\\
\frac{1}{\int_{-\infty}^{+\infty}\eta'\theta'dz}\nabla_x d^{(0)}\cdot\nabla_x\Big(\Delta_x d^{(0)}\int_{-\infty}^{+\infty}(\theta')^2dz+\int_{-\infty}^{+\infty}\widetilde{\mu}^{(0)}\theta'dz\Big),\quad &\text{on} \ \ \Gamma.
\end{array}\label{definition:l0}
\right.
\end{eqnarray}

%

\subsection{Boundary layer expansion in $\partial\Omega(\delta)$ }\label{section:bdexpan}

\indent

Let $d_B(x)<0$ be the signed distance from $x$ to $\partial\Omega$ and  $z=\frac{d_B(x)}{\varepsilon}\in(-\infty,0]$. In $\partial\Omega(\delta)=\{x\in\overline{\Omega}:-\delta<d_B(x)\leq0\}$, we  set
\begin{align}
u^\varepsilon(x,t)=u_B^\varepsilon(x,t,z)&=u_B^{(0)}(x,t,z)+\varepsilon u_B^{(1)}(x,t,z)+\varepsilon^2 u_B^{(2)}(x,t,z)+\cdots,\label{bd-1}\\
\mu^\varepsilon(x,t)=\mu_B^\varepsilon(x,t,z)&=\mu_B^{(0)}(x,t,z)+\varepsilon \mu_B^{(1)}(x,t,z)+\varepsilon^2 \mu_B^{(2)}(x,t,z)+\cdots,\label{bd-2}\\
\sigma^\varepsilon(x,t)=\sigma_B^\varepsilon(x,t,z)&=\sigma_B^{(0)}(x,t,z)+\varepsilon \sigma_B^{(1)}(x,t,z)+\varepsilon^2 \sigma_B^{(2)}(x,t,z)+\cdots,\label{bd-3}
\end{align}
and the following boundary-outer matching conditions: there exists a fixed positive constant $\nu$  such  that as $z\rightarrow-\infty$   there hold
\begin{align}
&D_x^\alpha D_t^\beta D_z^\gamma\Big(u_B^{(i)}(x,t,z)- u^{(i)}_{+}(x,t)\Big)=O(e^{-\nu|z|}),\label{bmc-1}\\
&D_x^\alpha D_t^\beta D_z^\gamma\Big(\mu_B^{(i)}(x,t,z) -\mu^{(i)}_{+}(x,t)\Big)=O(e^{-\nu|z|}),\label{bmc-2}\\
&D_x^\alpha D_t^\beta D_z^\gamma\Big(\sigma_B^{(i)}(x,t,z)- \sigma^{(i)}_{+}(x,t)\Big)=O(e^{-\nu|z|}),\label{bmc-3}
\end{align}
for $(x,t)\in\partial\Omega(\delta)\times[0,T]$ and $0\leq\alpha,\beta,\gamma\leq N$ with $N$ depending on actual expansion order.

Using Taylor expansion and (\ref{bd-1}) one has
\begin{align*}
f'(u_B^\varepsilon)=f'(u_B^{(0)})+\varepsilon f''(u_B^{(0)})u_B^{(1)}+\cdots+\varepsilon^k\Big(f''(u_B^{(0)})u_B^{(k)}+g_{_B}\big(u_B^{(0)},\cdots,u_B^{(k-1)}\big)\Big)+\cdots,
\end{align*}
here the function $g_{_B}\big(u_B^{(0)},\cdots,u_B^{(k-1)}\big)$ depends on $u_B^{(0)},\cdots,u_B^{(k-1)}.$
%

 We write (\ref{model}) in $\partial\Omega(\delta)$ in the new coordinate $(x,t,z)$ as follows
\begin{align}
&-\partial_{zz}\mu_B^\varepsilon-\varepsilon\Big(2\nabla_x\partial_z\mu_B^\varepsilon\cdot \nabla_x d_B+\partial_z\mu_B^\varepsilon \Delta_x d_B \Big)+\varepsilon^2\Big(\partial_tu_B^\varepsilon\nonumber
\\&\qquad\qquad\qquad\qquad\qquad\qquad\qquad\qquad-\Delta_x \mu_B^\varepsilon-\big(2\sigma_B^\varepsilon+u_B^\varepsilon-\mu_B^\varepsilon\big)\Big)=0,\label{bdequ-1}\\
&-\partial_{zz}\sigma_B^\varepsilon-\varepsilon\Big(2\nabla_x\partial_z\sigma_B^\varepsilon\cdot\nabla_x d_B+\partial_z\sigma_B^\varepsilon\Delta_x d_B \Big)+\varepsilon^2\Big(\partial_t\sigma_B^\varepsilon\nonumber
\\&\qquad\qquad\qquad\qquad\qquad\qquad\qquad\qquad-\Delta_x \sigma_B^\varepsilon+2\sigma_B^\varepsilon+u_B^\varepsilon-\mu_B^\varepsilon\Big)=0,\label{bdequ-2}\\
&-\partial_{zz}u_B^\varepsilon+f'(u_B^\varepsilon)-\varepsilon\big(2\nabla_x\partial_zu_B^\varepsilon\cdot\nabla_x d_B+\partial_zu_B^\varepsilon\Delta_x d_B+\mu_B^\varepsilon\big)-\varepsilon^2\Delta_x u_B^\varepsilon=0.\label{bdequ-3}
\end{align}

Moreover, homogeneous Neumann boundary conditions in (\ref{model}) imply on $\partial\Omega\times[0,T]$
\begin{align}
&\partial_{z}\mu_B^\varepsilon(x,t,0)+\varepsilon\nabla_x\mu_B^\varepsilon(x,t,0)\cdot \nabla_x d_B(x,t)=0,\label{bdc-1}\\
&\partial_{z}\sigma_B^\varepsilon(x,t,0)+\varepsilon\nabla_x\sigma_B^\varepsilon(x,t,0)\cdot \nabla_x d_B(x,t)=0,\label{bdc-2}\\
&\partial_{z}u_B^\varepsilon(x,t,0)+\varepsilon\nabla_xu_B^\varepsilon(x,t,0)\cdot \nabla_x d_B(x,t)=0.\label{bdc-3}
\end{align}

Firstly substituting (\ref{bd-1})-(\ref{bd-3}) into (\ref{bdequ-1})-(\ref{bdequ-3}) and (\ref{bdc-1})-(\ref{bdc-3}) and collecting all terms of  zero order we have
\begin{align*}
&-\partial_{zz}\mu_B^{(0)}=0,\\
&-\partial_{zz}\sigma_B^{(0)}=0,\\
&-\partial_{zz}u_B^{(0)}+f'(u_B^{(0)})=0,
\end{align*}
and on $\partial\Omega\times[0,T]$
\begin{align*}
&\partial_{z}\mu_B^{(0)}(x,t,0)=0,\\
&\partial_{z}\sigma_B^{(0)}(x,t,0)=0,\\
&\partial_{z}u_B^{(0)}(x,t,0)=0.
\end{align*}

%
Therefore we can take
\begin{align}
\mu_B^{(0)}=\mu^{(0)}_{+},\ \ \ \sigma_B^{(0)}=\sigma^{(0)}_{+},\ \ \ u_B^{(0)}=u^{(0)}_{+}=1,\label{definition:0bd}
\end{align}
which satisfy (\ref{bmc-1})-(\ref{bmc-3}) in the case of $i=0.$

Next substituting (\ref{bd-1})-(\ref{bd-3}) into (\ref{bdequ-1})-(\ref{bdequ-3}) and (\ref{bdc-1})-(\ref{bdc-3}) and collecting all terms of  $\varepsilon$-order we have
\begin{align*}
&-\partial_{zz}\mu_B^{(1)}=0,\\
&-\partial_{zz}\sigma_B^{(1)}=0,\\
&-\partial_{zz}u_B^{(1)}+f''(1)u_B^{(1)}=\mu_B^{(0)},
\end{align*}
and  on $\partial\Omega\times[0,T]$
\begin{align*}
&\partial_{z}\mu_B^{(1)}(x,t,0)=-\frac{\partial\mu^{(0)}_{+}}{\partial\nn},\\
& \partial_{z}\sigma_B^{(1)}(x,t,0)=-\frac{\partial\sigma^{(0)}_{+}}{\partial\nn},\\
&\partial_{z}u_B^{(1)}(x,t,0)=0,
\end{align*}
here $\nn$ is the unit outer normal to $\partial\Omega$.

Therefore we can take
\begin{align*}
\mu_B^{(1)}=\mu^{(1)}_{+},\ \ \ \sigma_B^{(1)}=\sigma^{(1)}_{+},\ \ \ u_B^{(1)}=u^{(1)}_{+},
\end{align*}
which satisfy (\ref{bmc-1})-(\ref{bmc-3}) in the case of $i=1$ and imply  on $\partial\Omega\times[0,T]$
\begin{align}
\frac{\partial\mu^{(0)}_{+}}{\partial\nn}=\frac{\partial\sigma^{(0)}_{+}}{\partial\nn}=0.\label{bdc-11-0}
\end{align}

\subsection{Solving the leading order terms}
\indent

 $u^{(0)}_{\pm}$ and $\widetilde{u}^{(0)}$ are determined by  (\ref{out-leading}) and (\ref{in-leading}) respectively. Collecting (\ref{out-leading-1}), (\ref{out-leading-2}), (\ref{jump condition:mu0}), (\ref{jump condition:mun0}), (\ref{jump condition:sigman0}), (\ref{curvature condition}) and (\ref{bdc-11-0}) one has
\begin{eqnarray*}
\left \{
\begin {array}{ll}
-\Delta \mu^{(0)}_{\pm}+\mu^{(0)}_{\pm}=2\sigma^{(0)}_{\pm}\pm1,\ \qquad\qquad\qquad&\text{in}\ \ \Omega_{\pm},\\[5pt]
\partial_t\sigma^{(0)}_{\pm}-\Delta \sigma^{(0)}_{\pm}+2\sigma^{(0)}_{\pm}=\mu^{(0)}_{\pm}\mp1,\ &\text{in}\ \ \Omega_{\pm},\\[5pt]
\left[\mu^{(0)}\right]=[\sigma^{(0)}]=0,\ &\text{on}\ \ \Gamma,\\[5pt]
\big[\frac{\partial\mu^{(0)}}{\partial\nn}\big]=-2V,&\text{on}\ \ \Gamma,\\[5pt]
\big[\frac{\partial\sigma^{(0)}}{\partial\nn}\big]=0,&\text{on}\ \ \Gamma,\\[5pt]
\mu^{(0)}_{\pm}=\kappa\int_{-1}^{1}\sqrt{2f(u)}du,&\text{on}\ \ \Gamma,\\[5pt]
\frac{\partial \mu_+^{(0)}}{\partial\nn}=\frac{\partial \sigma_+^{(0)}}{\partial\nn}=0,\quad &\text{on}\ \partial\Omega,
\end{array}
\right.
\end{eqnarray*}
which is just the sharp interface model (\ref{sharp interface model}) of (\ref{model}).  Therefore we recover the sharp interface model by the above matched asymptotic expansion method.

Let (\ref{sharp interface model}) has a local smooth solution $(\mu,\sigma,\Gamma)$ which satisfies  (\ref{positive distance}). Let  $\mu^{(0)}_{\pm}=\mu\big|_{\Omega_{\pm}}, \sigma^{(0)}_{\pm}=\sigma\big|_{\Omega_{\pm}}$ and $d^{(0)}$ be the signed distance to $\Gamma$. Then
$\Gamma=\{(x,t)\in\Omega\times(0,T):d^{(0)}(x,t)=0\}$ and $\Omega_{\pm}=\{(x,t)\in\Omega\times(0,T):d^{(0)}(x,t)\gtrless 0\}$.
$p^{(0)},q^{(0)},g^{(0)},h^{(0)}$ and $l^{(0)}$ are determined by (\ref{definition:p0}), (\ref{definition:q0}), (\ref{definition:g0}), (\ref{definition:h0}) and (\ref{definition:l0}) respectively. $\widetilde{\mu}^{(0)}$ and $\widetilde{\sigma}^{(0)}$ are determined by (\ref{definition:mu0-in}) and (\ref{definition:sigma0-in}) respectively. And $\mu^{(0)}_{B}, \sigma^{(0)}_{B},u^{(0)}_{B}$ are determined by (\ref{definition:0bd}).  Moreover, the inner-outer matching conditions (\ref{mc-1})-(\ref{mc-3}) and  the boundary-outer matching conditions (\ref{bmc-1})-(\ref{bmc-3}) hold   for $i=0.$

\begin{Remark}
We can extend $(u^{(0)}_{\pm}, \mu^{(0)}_{\pm}, \sigma^{(0)}_{\pm})$ smoothly from $\Omega_{\pm}$ to $\Omega$ as in Remark 4.1 in  \cite{ABC}.
\end{Remark}
\indent

\section{Spectral condition and error estimates}
\indent

For clarity we leave the higher order expansions and the construction of the approximate solution $(u^A, \mu^A,\sigma^A, \varphi^A)$ which is defined in (\ref{approximate solutions}) to the last two sections.
In this section we firstly establish a spectral condition and  give a  smallness estimation of the error between the approximate solution and the true solution. Then Theorem \ref{main theorem} and  Corollary \ref{main corollary} will  be proved.
\subsection{Spectral condition}
\indent

%
\begin{Theorem}\label{spectral theorem} (Spectral condition) There exist two positive constants $\varepsilon_0$ and $C$ such that for any $0<\varepsilon<\varepsilon_0$, $v\in H^1(\Omega)$ and $w\in H^2(\Omega)$ with $\Delta w=v$ there holds
\begin{align}
\int_{\Omega}\bigg(\varepsilon|\nabla v|^2+\frac{1}{\varepsilon}f''(u^A)v^2\bigg)dx\geq -C\int_{\Omega}|\nabla w|^2dx,\label{spectral}
\end{align}
where $u^A$ is  an approximate solution of $u^\varepsilon$ which satisfies (\ref{profile of app}) and will be  constructed in Section \ref{onstruction of an approximate solution}.
\end{Theorem}

Thanks to Theorem 3.1 in \cite{C}  we only need to prove the following lemma.
\begin{Lemma} \label{spectral lemma}
There exist two positive constants $\varepsilon_0$ and $C$ such that for any $0<\varepsilon<\varepsilon_0$ and $v\in H^1(\Omega)$  there holds
\begin{align}
\int_{\Omega}\bigg(|\nabla v|^2+\frac{1}{\varepsilon^2}f''(u^A)v^2\bigg)dx\geq -C\int_{\Omega}v^2dx.\label{spectral-0}
\end{align}
\end{Lemma}
In fact, if
\begin{align*}
\int_{\Omega}\bigg(|\nabla v|^2+\frac{1}{\varepsilon^2}f''(u^A)v^2\bigg)dx\geq 0,
\end{align*}
then (\ref{spectral}) holds obviously. If
\begin{align*}
\int_{\Omega}\bigg(|\nabla v|^2+\frac{1}{\varepsilon^2}f''(u^A)v^2\bigg)dx\leq 0,
\end{align*}
and $w\in H^2(\Omega)$ with $\Delta w=v$, then the assumptions of  Theorem 3.1 in \cite{C} hold, hence we have
\begin{align*}
-\varepsilon\int_{\Omega}v^2dx\geq-C\int_{\Omega}|\nabla w|^2dx
\end{align*}
which and (\ref{spectral-0}) immediately imply (\ref{spectral})

Theorem \ref{spectral theorem}  and Lemma  \ref{spectral lemma}  have been used to prove the convergence of the Cahn-Hilliard equation to the Hele-Shaw model in \cite{ABC}. However the proof of Lemma  \ref{spectral lemma} were   established in another paper \cite{C}. Here we write the details in
a concise way which will show clearly how to vanish the singularity.

In order to prove Lemma  \ref{spectral lemma} we show the following proposition on the spectral analyses of  Neumann eigenvalue problem   (\ref{operator:f}). The proof has been given in \cite{C} and proved by a new method in \cite{FWZZ}.
\begin{Proposition}(\cite{C})\label{Neumann eigenvalue}
(1)(Estimate of the first eigenvalue of $\mathcal{L}_f$)
\begin{align}
\lambda^f_1\triangleq \inf_{\|q\|=1}\int_{I_{\ve}}\Big(\big(q')^2+f''(\theta)q^2\Big)dz=O(e^{-\frac{C}{\ve}})
,\label{eigen:f-1}
\end{align}
here $C$ is a positive constant independent of small $\ve$ and  $\|q\|=\big(\int_{I_{\ve}} q^2(z)dz\big)^{\frac{1}{2}}$.

(2)(Estimate of the second eigenvalue of $\mathcal{L}_f$)
\begin{align}
\lambda^f_2\triangleq \inf_{\|q\|=1,q\bot q^f_1}\int_{I_{\ve}}\Big(\big(q')^2+f''(\theta)q^2\Big)dz\geq c_f>0
,\label{eigen:f-2}
\end{align}
here $q\bot q^f_1\Leftrightarrow \int_{I_{\ve}} qq^f_1dz=0$, $q^f_1$ is the normalized eigenfunction corresponding to $\lambda^f_1$ and $c_f$ is a positive constant independent of small $\ve$.

(3)(Characterization of the first normalized eigenfunction of $\mathcal{L}_f$)
\begin{align}
\|q^f_1-\alpha \theta'\|^2=O(e^{-\frac{C}{\ve}}),\label{eigenfunc:f}
\end{align}
here $\alpha=\frac{1}{\| \theta'\|}$.
\end{Proposition}

Let
\begin{align*}
d^{[k]}(x,t)=d^{(0)}(x,t)+\varepsilon d^{(1)}(x,t)+\varepsilon^2 d^{(2)}(x,t)+\cdots+\varepsilon^k d^{(k)}(x,t),
\end{align*}
then $d^{[k]}(x,t)(k\geq1)$ is a $k$-th approximate of the signed distance  from $x$ to the interface $\Gamma_k^\ve\triangleq\{(x,t):d^{[k]}(x,t)=0\}$ in the following sense.

\begin{Lemma}\label{approximate of the signed distance}For every fixed $t\in[0,T]$, let $r_t(x)$  be the signed distance from $x$ to $\Gamma_k^\ve$. Then  for small $\varepsilon$
\begin{align}
\big\|r_t(x)-d^{[k]}(x,t)\big\|_{C^1(\Gamma_t(\delta))}=O(\varepsilon^{k+1}).\label{approximate distance}
\end{align}
\end{Lemma}
\begin{proof}
  Noting that $|\nabla d^{[k]}|^2=1+O(\varepsilon^{k+1})$, then for small $\varepsilon$ one gets
  \begin{align*}
|\nabla d^{[k]}|-1=\frac{|\nabla d^{[k]}|^2-1}{|\nabla d^{[k]}|+1}=O(\varepsilon^{k+1}).
\end{align*}

Since $r_t(x)$  is the signed distance, then $|\nabla r_t|=1$ and $\nabla r_t$ is parallel to $\nabla d^{[k]}$. Consequently, we obtain
\begin{align*}
|\nabla r_t(x)-\nabla d^{[k]}(x,t)|^2&=|\nabla r_t(x)|^2-2\nabla d^{[k]}(x,t)\cdot\nabla r_t(x)+|\nabla d^{[k]}(x,t)|^2
\\&=1-2|\nabla d^{[k]}(x,t)|+|\nabla d^{[k]}(x,t)|^2
\\&=\big(1-|\nabla d^{[k]}(x,t)|\big)^2
\\&=O(\varepsilon^{2(k+1)}).
\end{align*}

Choosing $x_0\in\Gamma_k^\ve$, i.e., $r_t(x_0)=d^{[k]}(x_0,t)=0$, then
\begin{align*}
|r_t(x)-d^{[k]}(x,t)|&=|r_t(x)-d^{[k]}(x,t)-r_t(x_0)+d^{[k]}(x_0,t)|
\\&=\bigg|\int_0^1\Big(\nabla r_t(t'x+(1-t')x_0)-\nabla d^{[k]}\big((t'x+(1-t')x_0),t\big)\Big)\cdot(x-x_0)dt'\Big|
\\&\leq C\varepsilon^{k+1}.
\end{align*}

Hence we complete the proof of this lemma.
\end{proof}

Let $s_t(x)$ be the projection of $x$ on $\Gamma_k^\ve$ along the normal of $\Gamma_k^\ve$. Then the transformation $x\longmapsto(r_t(x),s_t(x))$ is a diffeomorphism in $\Gamma_k^\ve(\delta)$ for small $\delta$. Let $J(r_t,s_t)=\det\frac{\partial x^{-1}(r_t,s_t)}{\partial(r_t,s_t)}$ be the Jacobian of the transformation, then $J|_{\Gamma_k^\ve}=1$ and $\frac{\partial J}{\partial r_t}|_{\Gamma_k^\ve}=0.$ Thus
\begin{align}
0<C_1\leq J(r_t,s_t)\leq C_2,\ \ \ \bigg|J_{r_t}(r_t,s_t)\triangleq\frac{\partial J}{\partial r_t}(r_t,s_t)\bigg|\leq C\big|r_t\big|.\label{Jacobic derivative}
\end{align}

In view of (\ref{in-1-3}) and the similar arguments in P199 in \cite{ABC} we obtain
\begin{Lemma}\label{u1:rewrite}
In $\Gamma(\delta)$, $\widetilde{u}^{(1)}$ can be expressed as
\begin{align*}
&\widetilde{u}^{(1)}(x,t,z)\bigg|_{z=\frac{r_t(x)}{\varepsilon}}=\bar{p}(s_t(x))\theta_1\bigg(\frac{r_t(x)}{\varepsilon}\bigg)+\bar{q}(x)=\bar{p}(s_t(x))\theta_1(z)+\bar{q}(x),
\end{align*}
where $\theta_1\in L^\infty(\mathbb{R})$, $\bar{p}\in L^\infty(\Gamma(\delta))$ and
\begin{align*}
\int_{-\infty}^{+\infty}f'''\big(\theta(z)\big)\theta_1(z)\big(\theta'(z)\big)^2dz=0,\ \ \ \big|\bar{q}(x)\big|\leq C(\varepsilon+\big|r_t(x)\big|)\leq C\varepsilon (1+|z|).
\end{align*}
\end{Lemma}

Now we focus on the proof of Lemma \ref{spectral lemma}.

{\it Proof of Lemma \ref{spectral lemma}.}
For clarity we divide into three steps to proceed.

\underline{Step 1.}
Noting that $f''(\pm 1)>0$, then for small $\varepsilon$ there holds
\begin{align*}
&\int_{\Omega}\bigg(|\nabla v|^2+\frac{1}{\varepsilon^2}f''(u^A)v^2\bigg)dx
\\&\geq\int_{\Gamma_k^\ve(\delta)}\bigg(|\nabla v|^2+\frac{1}{\varepsilon^2}f''(u^A)v^2\bigg)dx
\\&=\varepsilon^{-2}\int_{\Gamma_k^\ve}\int_{-\frac{\delta}{\varepsilon}}^{\frac{\delta}{\varepsilon}}\Big(|\partial_z v|^2+f''(u^A)v^2\Big)J(r_t(x),s_t(x))dzdS
\\&\geq\varepsilon^{-2}\int_{\Gamma_k^\ve}\int_{-\frac{\delta}{\varepsilon}}^{\frac{\delta}{\varepsilon}}\Big(|\partial_z v|^2+f''\big(\theta(z)\big)v^2\Big)J(r_t(x),s_t(x))dzdS
\\&\quad+\varepsilon^{-1}\int_{\Gamma_k^\ve}\int_{-\frac{\delta}{\varepsilon}}^{\frac{\delta}{\varepsilon}} f'''\big(\theta(z)\big)\widetilde{u}^{(1)}(x,t,z)v^2J(r_t(x),s_t(x))dzdS-C\int_{\Omega}v^2dx,
\end{align*}
where we have used the fact that in $\Gamma_k^\ve(\delta)$
\begin{align*}
f''(u^A)&=f''\Big(\theta(z)+\varepsilon\widetilde{u}^{(1)}(x,t,z)+O(\varepsilon^2)\Big)\Big|_{z=\frac{d^{[k]}}{\varepsilon}}
\\&=f''\big(\theta(z)\big)+\varepsilon f'''\big(\theta(z)\big)\widetilde{u}^{(1)}(x,t,z)+O(\varepsilon^2)\Big|_{z=\frac{d^{[k]}}{\varepsilon}}
\\&=f''\big(\theta(z)\big)+\varepsilon f'''\big(\theta(z)\big)\widetilde{u}^{(1)}(x,t,z)+O(\varepsilon^2)\Big|_{z=\frac{r_t(x)}{\varepsilon}}.
\end{align*}
Set $\hat{v}=vJ^{\frac{1}{2}}$, from Lemma 5.8 in \cite{FWZZ} (we show the proof in Appendix for completeness of this paper) one gets
\begin{align}
&\int_{\Gamma_k^\ve}\int_{-\frac{\delta}{\varepsilon}}^{\frac{\delta}{\varepsilon}}\Big(|\partial_z v|^2+f''\big(\theta(z)\big)v^2\Big)J(r_t(x),s_t(x))dzdS
\nonumber\\&\geq\frac{3}{4}\int_{\Gamma_k^\ve}\int_{-\frac{\delta}{\varepsilon}}^{\frac{\delta}{\varepsilon}}\Big(|\partial_z \hat{v}|^2+f''\big(\theta(z)\big)\hat{v}^2\Big)dzdS-C\varepsilon^2\int_{\Omega}v^2dx.\label{Jacobic transform}
\end{align}

Consequently
\begin{align}
\int_{\Omega}\bigg(|\nabla v|^2+\frac{1}{\varepsilon^2}f''(u^A)v^2\bigg)dx
\geq\varepsilon^{-2}\int_{\Gamma_k^\ve}I dS+\varepsilon^{-1}\int_{\Gamma_k^\ve}IIdS-C\int_{\Omega}v^2dx,\label{modified bad terms}
\end{align}
where
\begin{align*}
I=\frac{3}{4}\int_{-\frac{\delta}{\varepsilon}}^{\frac{\delta}{\varepsilon}}\Big(|\partial_z \hat{v}|^2+f''\big(\theta(z)\big)\hat{v}^2\Big)dz,\
II=\int_{-\frac{\delta}{\varepsilon}}^{\frac{\delta}{\varepsilon}} f'''\big(\theta(z)\big)\widetilde{u}^{(1)}(x,t,z)\hat{v}^2dz.
\end{align*}

\underline{Step 2.} To deal with term $I$, we decompose  $\hat{v}=\gamma q_1^f + p_1$,  here $p_1\bot q^f_1$  and then $\|\hat{v}\|^2=\gamma^2+\|p_{1}\|^2$. Then
\begin{align*}
I&=\frac{3}{4}\int_{-\frac{\delta}{\varepsilon}}^{\frac{\delta}{\varepsilon}}\Big(|\partial_z \hat{v}|^2+f''\big(\theta(z)\big)\hat{v}^2\Big)dz
\nonumber\\&=\frac34\gamma^2\int_{-\frac{\delta}{\varepsilon}}^{\frac{\delta}{\varepsilon}}\Big(\big((q^f_{1})'\big)^2+f''\big(\theta(z)\big)(q^f_{1})^2\Big)dz
+\frac 34\int_{-\frac{\delta}{\varepsilon}}^{\frac{\delta}{\varepsilon}}
\Big((\partial_z{p_1})^2+f''\big(\theta(z)\big){p_1}^2\Big)dz\nonumber
\nonumber\\&\geq\frac 34\gamma^2\lambda^f_1+\frac 34\lambda^f_2\|{p_1}\|^2\geq-C\ve^2\gamma^2+\frac 34\lambda^f_2\|{p_1}\|^2
\nonumber\\&\geq-C\ve^2\int_{-\frac{\delta}{\varepsilon}}^{\frac{\delta}{\varepsilon}}v^2Jdz+\frac 34\lambda^f_2\|{p_1}\|^2,
\end{align*}
here we have used (\ref{eigen:f-1}) and (\ref{eigen:f-2}).

\underline{Step 3.} To estimate $II$, we use Lemma \ref{u1:rewrite}
and then
\begin{align*}
&\int_{-\frac{\delta}{\varepsilon}}^{\frac{\delta}{\varepsilon}}f'''\big(\theta(z)\big)\widetilde{u}^{(1)}(x,t,z)(\theta'(z))^2dz\nonumber\\&=\bar{p}(s_t(x))\int_{-\frac{\delta}{\varepsilon}}^{\frac{\delta}{\varepsilon}}f'''\big(\theta(z)\big)\theta_1(z)(\theta'(z))^2dz
+\int_{-\frac{\delta}{\varepsilon}}^{\frac{\delta}{\varepsilon}}f'''\big(\theta(z)\big)\bar{q}(x)(\theta'(z))^2dz
\nonumber\\&=-\bar{p}(s_t(x))\int_{-\infty}^{-\frac{\delta}{\ve}}f'''\big(\theta(z)\big)\theta_1(z)(\theta'(z))^2dz
-\bar{p}(s_t(x))\int_{\frac{\delta}{\ve}}^{+\infty}f'''\big(\theta(z)\big)\theta_1(z)(\theta'(z))^2dz
\nonumber\\&\quad+\int_{-\frac{\delta}{\varepsilon}}^{\frac{\delta}{\varepsilon}}f'''\big(\theta(z)\big)\bar{q}(x)(\theta'(z))^2dz
\nonumber\\&=O(\varepsilon).
\end{align*}
Hence we find
\begin{align*}
II&=\int_{-\frac{\delta}{\varepsilon}}^{\frac{\delta}{\varepsilon}} f'''\big(\theta(z)\big)\widetilde{u}^{(1)}(x,t,z)\hat{v}^2dzdS
\nonumber\\&=\gamma^2\int_{-\frac{\delta}{\varepsilon}}^{\frac{\delta}{\varepsilon}}f'''\big(\theta(z)\big)\widetilde{u}^{(1)}(q^f_{1})^2dz+ 2\gamma\int_{-\frac{\delta}{\varepsilon}}^{\frac{\delta}{\varepsilon}}f'''\big(\theta(z)\big)\widetilde{u}^{(1)}q^f_{1}{p_1} dz
\nonumber\\&\qquad\qquad\qquad+\int_{-\frac{\delta}{\varepsilon}}^{\frac{\delta}{\varepsilon}}f'''\big(\theta(z)\big)\widetilde{u}^{(1)}{p_1}^2dz
\nonumber\\&=\gamma^2
\alpha^2\int_{-\frac{\delta}{\varepsilon}}^{\frac{\delta}{\varepsilon}}f'''\big(\theta(z)\big)\widetilde{u}^{(1)}(\theta')^2dz
+O(e^{-\frac{C}{\ve}})\gamma^2+O(1)\|{p_1}\|\gamma+O(1)\|{p_1}\|^2
\nonumber\\&=O(\ve)\gamma^2
+O(e^{-\frac{C}{\ve}})\gamma^2+O(1)\gamma\|{p_1}\|+O(1)\|{p_1}\|^2
\nonumber\\&\geq\big(O(e^{-\frac{C}{\ve}})+O(\ve)\big)\gamma^2
-\frac{1}{8\varepsilon}\lambda^f_2\|{p_1}\|^2
\nonumber\\&\geq-C\ve\gamma^2-\frac{1}{8\varepsilon}\lambda^f_2\|{p_1}\|^2
\nonumber\\&\geq-C\ve\int_{-\frac{\delta}{\varepsilon}}^{\frac{\delta}{\varepsilon}}v^2Jdz
-\frac{1}{8\varepsilon}\lambda^f_2\|{p_1}\|^2,
\end{align*}
where we have used $\hat{v}=\gamma q_1^f + p_1$ and (\ref{eigenfunc:f}).

Therefore we have
\begin{align*}
\varepsilon^{-2}\int_{\Gamma_k^\ve}I dS+\varepsilon^{-1}\int_{\Gamma_k^\ve}IIdS\geq-C\int_{\Omega}v^2dx
\end{align*}
which and (\ref{modified bad terms}) imply  the desired conclusion.

Thus we complete the proof of the lemma.

\subsection{Error estimates}
\indent

Let $u^{err}=u^\varepsilon-u^A,\mu^{err}=\mu^\varepsilon-\mu^A, \sigma^{err}=\sigma^\varepsilon-\sigma^A, \varphi^{err}=\varphi^\varepsilon-\varphi^A $ with $\varphi^\varepsilon=u^\varepsilon+\sigma^\varepsilon$
and impose
\begin{align}
u_0^\varepsilon(x)=\varphi^A(x,0)-\sigma_0^\varepsilon(x),\ \sigma_0^\varepsilon(x)=\sigma^A(x,0).\label{initial data-impose}
\end{align}
Here  $(u^A, \mu^A,\sigma^A, \varphi^A)$  is defined in (\ref{approximate solutions}).

By (\ref{model}) and (\ref{approximate model}) there hold
\begin{eqnarray}
\left \{
\begin {array}{ll}
\partial_t\varphi^{err}-\Delta \mu^{err}-\Delta \sigma^{err}=0,\quad&\text{in}\ \ \Omega\times(0,T),\\[3pt]
\partial_t\sigma^{err}-\Delta \sigma^{err}=-(2\sigma^{err}+u^{err}-\mu^{err}),\ &\text{in}\ \ \Omega\times(0,T),\\[3pt]
\mu^{err}=-\varepsilon \Delta u^{err}+\frac{1}{\varepsilon}f''(u^A)u^{err}+\frac{1}{\varepsilon}\mathbb{F}-\omega_4^A,\ &\text{in}\ \ \Omega\times(0,T),\\[3pt]
u^{err}=\varphi^{err}-\sigma^{err}+\omega_5^A,\ &\text{in}\ \ \Omega\times(0,T),\\[3pt]
\varphi^{err}(x,0)=\sigma^{err}(x,0)= 0,&\text{on}\ \ \Omega\times\{0\},\\[3pt]
\frac{\partial \varphi^{err}}{\partial\nn}=\frac{\partial \mu^{err}}{\partial\nn}=\frac{\partial \sigma^{err}}{\partial\nn}=0,\quad &\text{on}\ \partial\Omega\times(0,T),\label{err equation}
\end{array}
\right.
\end{eqnarray}
where
\begin{align*}
\mathbb{F}=f'(u^{err}+u^A)-f'(u^A)-f''(u^A)u^{err}=4(u^{err})^3+8u^A(u^{err})^2
\end{align*}
and
\begin{align*}
\omega_5^A=-\frac{1}{|\Omega|}\int_0^t\int_{\Omega}\big(\omega^A_1+\omega^A_2\big)(x,t')dxdt'+\omega_2^A+\tilde{\mu}^A=O(\varepsilon^{k-1}).
\end{align*}

%
%

\begin{Theorem} For small $\varepsilon$ and large $k$, there exist  $\gamma=\gamma(k)\in(1,k)$  which is an increasing function of $k$ and  a positive constant $C$ such that
\begin{align*}
\big\|u^{err}\big\|_{L^p(\Omega\times(0.T))}+\big\|\varphi^{err}\big\|_{L^p(\Omega\times(0.T))}+\big\|\sigma^{err}\big\|_{L^p(\Omega\times(0.T))}\leq C\varepsilon^{\gamma}.
\end{align*}
\end{Theorem}

\begin{proof}
For the sake of clarity we omit the superscript ``err" in the proof and only in this proof.

Noting that
\begin{align*}
\int_\Omega\varphi (t,x)dx=\int_0^t\int_\Omega\partial_t\varphi (t,x)dx=\int_0^t\int_\Omega(\Delta \mu+\Delta \sigma)(t,x)dx=0,
\end{align*}
then there exists a function $\psi(t,\cdot)(t\in(0,T))$ which satisfies
\begin{eqnarray}
\left \{
\begin {array}{ll}
-\Delta\psi=\varphi,\ &\text{in}\ \ \Omega,\\[3pt]
\frac{\partial \psi}{\partial\nn}=0,\quad &\text{on}\ \partial\Omega,\\[3pt]
 \int_\Omega\psi(t,x) dx=0,&t\in(0,T),\label{inverse laplace-1}
\end{array}
\right.
\end{eqnarray}

And  let $\varrho(t,\cdot)(t\in(0,T))$ be the solution of the following equation
\begin{eqnarray}
\left \{
\begin {array}{ll}
-\Delta\varrho=\sigma,\ &\text{in}\ \ \Omega,\\
\varrho=0,\quad &\text{on}\ \partial\Omega.\label{inverse laplace-2}
\end{array}
\right.
\end{eqnarray}

Multiplying the first equation in (\ref{err equation}) by  $\psi$ and integrating by parts we have
\begin{align}
&\frac{1}{2}\frac{d}{dt}\int_{\Omega}|\nabla\psi|^2dx+\int_{\Omega}\bigg(\varepsilon|\nabla \varphi|^2+\frac{1}{\varepsilon}f''(u^A)\varphi^2\bigg)dx\nonumber\\&\quad-\int_{\Omega}\bigg(\varepsilon \nabla \varphi\nabla \sigma+\frac{1}{\varepsilon}f''(u^A)\varphi\sigma\bigg)dx+\frac{1}{\varepsilon}\int_{\Omega}\mathbb{F} \varphi dx+\int_{\Omega}\nabla\psi\nabla\sigma dx
\nonumber\\&=\int_{\Omega}\omega_6^A \varphi dx,\label{energy-1}
\end{align}
where
\begin{align*}
\omega_6^A=\omega_4^A+\varepsilon\Delta\omega_5^A-\varepsilon^{-1}f''(u^A)\omega_5^A=O(\varepsilon^{k-2}).
\end{align*}
%
%

Multiplying the second equation in (\ref{err equation}) by  $\sigma$ and integrating by parts we have
\begin{align}
&\frac{1}{2}\frac{d}{dt}\int_{\Omega}\sigma^2dx+\int_{\Omega}|\nabla \sigma|^2dx+\int_{\Omega}\sigma^2 dx+\int_{\Omega}\bigg(\varepsilon|\nabla \sigma|^2+\frac{1}{\varepsilon}f''(u^A)\sigma^2\bigg)dx\nonumber\\&\qquad\qquad-\int_{\Omega}\bigg(\varepsilon \nabla \varphi\nabla \sigma+\frac{1}{\varepsilon}f''(u^A)\varphi\sigma\bigg)dx-\frac{1}{\varepsilon}\int_{\Omega}\mathbb{F} \sigma dx+\int_{\Omega}\nabla\psi\nabla\sigma dx
\nonumber\\&=\int_{\Omega}\omega_7^A\sigma dx,\label{energy-2}
\end{align}
where
\begin{align*}
\omega_7^A=-\omega_4^A-\omega_5^A-\varepsilon\Delta\omega_5^A+\varepsilon^{-1}f''(u^A)\omega_5^A=O(\varepsilon^{k-2}).
\end{align*}

Combining (\ref{energy-1}) and (\ref{energy-2}) we get
\begin{align}
&\frac{1}{2}\frac{d}{dt}\int_{\Omega}|\nabla\psi|^2dx+\frac{1}{2}\frac{d}{dt}\int_{\Omega}\sigma^2dx+\int_{\Omega}|\nabla \sigma|^2dx+\int_{\Omega}\sigma^2 dx\nonumber\\&\qquad+\int_{\Omega}\bigg(\varepsilon|\nabla (\varphi-\sigma)|^2+\frac{1}{\varepsilon}f''(u^A)(\varphi-\sigma)^2\bigg)dx
\nonumber\\&\qquad+2\int_{\Omega}\nabla\psi\nabla\sigma dx+\frac{1}{\varepsilon}\int_{\Omega}(\varphi-\sigma)\mathbb{F}  dx
\nonumber\\&=\int_{\Omega}\omega_6^A \varphi dx+\int_{\Omega}\omega_7^A\sigma dx.\label{energy-3}
\end{align}

Moreover, we can easily find
\begin{align}
(\varphi-\sigma)\mathbb{F}&=4(\varphi-\sigma)^4+(8u^A+12\omega_5^A)(\varphi-\sigma)^3+(16u^A\omega_5^A+12(\omega_5^A)^2)(\varphi-\sigma)^2
\nonumber\\[3pt]
&\qquad\qquad\quad\ \ +(8u^A(\omega_5^A)^2+4(\omega_5^A)^3)(\varphi-\sigma)
\nonumber\\[3pt]
&\geq-\widetilde{C}_p|\varphi-\sigma|^p+\omega_8^A|\varphi-\sigma|,
\nonumber\\[3pt]
&\geq-C_p|\varphi|^p-C_p|\sigma|^p+\omega_8^A|\varphi-\sigma|,\ \ \forall p\in(1,3],\label{nonlinear estimate}
\end{align}
where the positive constants $\widetilde{C}_p, C_p$ depend on $p$ and $\omega_8^A=O(\varepsilon^{k-1})$.

Plugging (\ref{nonlinear estimate}) into (\ref{energy-3}),  using the Young's inequality and the Sobolev inequality we arrive at
\begin{align}
&\frac{1}{2}\frac{d}{dt}\int_{\Omega}|\nabla\psi|^2dx+\frac{1}{2}\frac{d}{dt}\int_{\Omega}\sigma^2dx+\frac{1}{2}\int_{\Omega}|\nabla \sigma|^2dx+\int_{\Omega}\sigma^2 dx
\nonumber\\&\qquad+\int_{\Omega}\bigg(\varepsilon|\nabla(\varphi-\sigma)|^2+\frac{1}{\varepsilon}f''(u^A)(\varphi-\sigma)^2\bigg)dx
\nonumber\\&\leq C\int_{\Omega}|\nabla\psi|^2dx+C\int_{\Omega}\sigma^2dx+\frac{C_p}{\varepsilon}\int_{\Omega}|\varphi|^pdx+\frac{C_p}{\varepsilon}\int_{\Omega}|\sigma|^pdx
\nonumber\\&\quad+\bigg(\int_{\Omega}|\varphi|^pdx\bigg)^{\frac{1}{p}}\bigg(\int_{\Omega}|\omega_9^A|^qdx\bigg)^{\frac{1}{q}}
\nonumber\\&\quad+\bigg(\int_{\Omega}|\sigma|^pdx\bigg)^{\frac{1}{p}}\bigg(\int_{\Omega}|\omega_{10}^A|^qdx\bigg)^{\frac{1}{q}},\label{energy-4}
\end{align}
where  $\omega_9^A=O(\varepsilon^{k-2}), \omega_{10}^A=O(\varepsilon^{k-2})$ and $\frac{1}{p}+\frac{1}{q}=1.$

According to the spectral condition Theorem \ref{spectral theorem} one has for small $\varepsilon$
\begin{align}
\int_{\Omega}\bigg(\varepsilon|\nabla(\varphi-\sigma)|^2+\frac{1}{\varepsilon}f''(u^A)(\varphi-\sigma)^2\bigg)dx\geq -C\int_{\Omega}|\nabla (\psi-\varrho)|^2dx,\label{energy-5}
\end{align}
where $C$ is a positive constant independent of $t.$

By the Poincar\'{e} inequality and (\ref{inverse laplace-2}) one get
\begin{align}
\int_{\Omega}|\nabla \varrho|^2dx\leq C\int_{\Omega}\sigma^2dx.\label{energy-6}
\end{align}

Plugging (\ref{energy-5})-(\ref{energy-6}) into (\ref{energy-4}) we obtain
\begin{align*}
&\frac{1}{2}\frac{d}{dt}\int_{\Omega}|\nabla\psi|^2dx+\frac{1}{2}\frac{d}{dt}\int_{\Omega}\sigma^2dx+\frac{1}{2}\int_{\Omega}|\nabla \sigma|^2dx+\int_{\Omega}\sigma^2 dx
\nonumber\\
&\leq C\int_{\Omega}|\nabla\psi|^2dx+C\int_{\Omega}\sigma^2dx+\frac{C_p}{\varepsilon}\int_{\Omega}|\varphi|^pdx+\frac{C_p}{\varepsilon}\int_{\Omega}|\sigma|^pdx
\nonumber\\&\quad+\bigg(\int_{\Omega}|\varphi|^pdx\bigg)^{\frac{1}{p}}\bigg(\int_{\Omega}|\omega_9^A|^qdx\bigg)^{\frac{1}{q}}
\nonumber\\&\quad+\bigg(\int_{\Omega}|\sigma|^pdx\bigg)^{\frac{1}{p}}\bigg(\int_{\Omega}|\omega_{10}^A|^qdx\bigg)^{\frac{1}{q}}.
\end{align*}

It follows from the Gronwall inequality that for any $t\in(0,T]$
\begin{align}
&\sup_{0\leq t'\leq t}\Big(\big\|\nabla\psi\big\|^2_{L^2(\Omega)}+\big\|\sigma\big\|^2_{L^2(\Omega)}\Big)
+\big\|\sigma\big\|^2_{L^2(\Omega\times(0,t))}+\big\|\nabla\sigma\big\|^2_{L^2(\Omega\times(0,t))}
\nonumber\\
&\leq C\Big(\varepsilon^{-1}\big\|\varphi\big\|^p_{L^p(\Omega\times(0,t))}+\varepsilon^{-1}\big\|\sigma\big\|^p_{L^p(\Omega\times(0,t))}
+\big\|\varphi\big\|_{L^p(\Omega\times(0,t))}\big\|\omega_9^A\big\|_{L^q(\Omega\times(0,t))}
\nonumber\\&\qquad\qquad+\big\|\sigma\big\|_{L^p(\Omega\times(0,t))}\big\|\omega_{10}^A\big\|_{L^q(\Omega\times(0,t))}\Big).
\label{energy estimate-1}
\end{align}

Furthermore,  by (\ref{energy-4}) and  (\ref{energy estimate-1}) one gets
\begin{align*}
&\big\|\nabla(\varphi-\sigma)\big\|^2_{L^2(\Omega\times(0.t))}\nonumber\\
&\leq\varepsilon^{-2}\bigg(-\int_0^t\int_{\Omega}f''(u^A)(\varphi-\sigma)^2dxdt'\bigg)
+ C\varepsilon^{-2}\Big(\big\|\varphi\big\|^p_{L^p(\Omega\times(0,t))}+\big\|\sigma\big\|^p_{L^p(\Omega\times(0,t))}\Big)
\nonumber\\[3pt]
&\quad+C\varepsilon^{-1}\Big(\big\|\varphi\big\|_{L^p(\Omega\times(0,t))}\big\|\omega_9^A\big\|_{L^q(\Omega\times(0,t))}
+\big\|\sigma\big\|_{L^p(\Omega\times(0,t))}\big\|\omega_{10}^A\big\|_{L^q(\Omega\times(0,t))}\Big)
\nonumber\\[3pt]
&\leq \varepsilon^{-2}\big\|\varphi-\sigma\big\|^2_{L^p(\Omega\times(0,t))}\big\|f''(
{u^A})\big\|_{L^\infty(\Omega\times(0,t))}\text{measure}\{f''({u^A})<0\}^{1-\frac{2}{p}}
\nonumber\\[3pt]
&\quad+ C\varepsilon^{-2}\Big(\big\|\varphi\big\|^p_{L^p(\Omega\times(0,t))}+\big\|\sigma\big\|^p_{L^p(\Omega\times(0,t))}\Big)
+C\varepsilon^{-1}\big\|\varphi\big\|_{L^p(\Omega\times(0,t))}\big\|\omega_7^A\big\|_{L^q(\Omega\times(0,t))}
\nonumber\\[3pt]
&\leq C\varepsilon^{-1-\frac{2}{p}}\Big(\big\|\varphi\big\|^2_{L^p(\Omega\times(0,t))}+\big\|\sigma\big\|^2_{L^p(\Omega\times(0,t))}\Big)
+ C\varepsilon^{-2}\Big(\big\|\varphi\big\|^p_{L^p(\Omega\times(0,t))}+\big\|\sigma\big\|^p_{L^p(\Omega\times(0,t))}\Big)
\nonumber\\[3pt]
&\quad+C\varepsilon^{-1}\Big(\big\|\varphi\big\|_{L^p(\Omega\times(0,t))}\big\|\omega_9^A\big\|_{L^q(\Omega\times(0,t))}
+\big\|\sigma\big\|_{L^p(\Omega\times(0,t))}\big\|\omega_{10}^A\big\|_{L^q(\Omega\times(0,t))}\Big).
\end{align*}

Thus
\begin{align*}
\big\|\nabla\varphi\big\|^2_{L^2(\Omega\times(0,t))}&\leq\big\|\nabla(\varphi-\sigma)\big\|^2_{L^2(\Omega\times(0,t))}+\big\|\nabla\sigma\big\|^2_{L^2(\Omega\times(0,t))}
\nonumber\\[3pt]
&\leq C\varepsilon^{-1-\frac{2}{p}}\Big(\big\|\varphi\big\|^2_{L^p(\Omega\times(0,t))}+\big\|\sigma\big\|^2_{L^p(\Omega\times(0,t))}\Big)
\nonumber\\[3pt]
&\quad+ C\varepsilon^{-2}\Big(\big\|\varphi\big\|^p_{L^p(\Omega\times(0,t))}+\big\|\sigma\big\|^p_{L^p(\Omega\times(0,t))}\Big)
\nonumber\\[3pt]
&\quad+C\varepsilon^{-1}\Big(\big\|\varphi\big\|_{L^p(\Omega\times(0,t))}\big\|\omega_9^A\big\|_{L^q(\Omega\times(0,t))}
+\big\|\sigma\big\|_{L^p(\Omega\times(0,t))}\big\|\omega_{10}^A\big\|_{L^q(\Omega\times(0,t))}\Big).
\end{align*}
%

Applying the Sobolev imbedding theorem and the H\"{o}lder inequality we get
\begin{align}
&\big\|\varphi\big\|^p_{L^p(\Omega\times(0,t))}+\big\|\sigma\big\|^p_{L^p(\Omega\times(0,t))}
\nonumber\\&=\int_0^t\Big(\big\|\varphi\big\|^p_{L^p(\Omega)}+\big\|\sigma\big\|^p_{L^p(\Omega)}\Big)(t')dt'
\nonumber\\&\leq C\int_0^t\big\|\varphi\big\|^{\theta p}_{L^2(\Omega)}\big\|\nabla\varphi\big\|^{(1-\theta) p}_{L^2(\Omega)}(t')dt'+C\int_0^t\big\|\sigma\big\|^{\theta p}_{L^2(\Omega)}\big\|\nabla\sigma\big\|^{(1-\theta) p}_{L^2(\Omega)}(t')dt'
\nonumber\\&\qquad\qquad\qquad\qquad\qquad\qquad\qquad\quad+C\int_0^t\big\|\sigma\big\|^{p}_{L^2(\Omega)}(t')dt'
\nonumber\\&\leq C\int_0^t\big\|\nabla\psi\big\|^{\frac{\theta p}{2}}_{L^2(\Omega)}\big\|\nabla\varphi\big\|^{(1-\frac{\theta}{2}) p}_{L^2(\Omega)}(t')dt'+C\int_0^t\big\|\sigma\big\|^{\theta p}_{L^2(\Omega)}\big\|\nabla\sigma\big\|^{(1-\theta) p}_{L^2(\Omega)}(t')dt'
\nonumber\\&\qquad\qquad\qquad\qquad\qquad\qquad\qquad\quad+C\int_0^t\big\|\sigma\big\|^{p}_{L^2(\Omega)}(t')dt'
\nonumber\\&\leq C\Big(\int_0^t\big\|\nabla\psi\big\|^{\frac{\alpha\theta p}{2}}_{L^2(\Omega)}(t')dt'\Big)^{\frac{1}{\alpha}}\big\|\nabla\varphi\big\|^{(1-\frac{\theta}{2}) p}_{L^2(\Omega\times(0,t))}+C\Big(\int_0^t\big\|\sigma\big\|^{\beta\theta p}_{L^2(\Omega)}(t')dt'\Big)^{\frac{1}{\beta}}\big\|\nabla\sigma\big\|^{(1-\theta) p}_{L^2(\Omega\times(0,t))}
\nonumber\\&\qquad\qquad\qquad\qquad\qquad\qquad\qquad\quad+C\int_0^t\big\|\sigma\big\|^{p}_{L^2(\Omega)}(t')dt'
\nonumber\\&\leq C\Big(\sup_{0\leq t'\leq t}\big\|\nabla\psi\big\|_{L^2(\Omega)}\Big)^{\frac{\theta p}{2}}\big\|\nabla\varphi\big\|^{(1-\frac{\theta}{2}) p}_{L^2(\Omega\times(0,t))}+C\Big(\sup_{0\leq t'\leq t}\big\|\sigma\big\|_{L^2(\Omega)}\Big)^{\theta p}\big\|\nabla\sigma\big\|^{(1-\theta) p}_{L^2(\Omega\times(0,t))}
\nonumber\\&\qquad\qquad\qquad\qquad\qquad\qquad\qquad\quad+CT\Big(\sup_{0\leq t'\leq t}\big\|\sigma\big\|_{L^2(\Omega)}\Big)^{p},\label{energy estimate-4}
\end{align}
where $\frac{1}{p}=\frac{\theta}{2}+\frac{(N-2)(1-\theta)}{2N}$, $\frac{1}{\alpha}+\frac{(1-\frac{\theta}{2}) p}{2}=1$ and $\frac{1}{\beta}+\frac{(1-\theta)p}{2}=1$.

Setting $\Theta(t)=\big\|\varphi\big\|_{L^p(\Omega\times(0,t))}+\big\|\sigma\big\|_{L^p(\Omega\times(0,t))}(t\in(0,T])$, then with the help of $(\ref{energy estimate-1})-(\ref{energy estimate-4})$, we derive a recursive inequality for $\Theta(t)$:
\begin{align}
\Theta^p(t)&\leq\Big(\varepsilon^{-1}\Theta^p(t)+\Theta(t)\big\|\omega_{11}^A\big\|_{L^q(\Omega\times(0,t))}\Big)^\frac{\theta p}{4}\Big(\varepsilon^{-1-\frac{2}{p}}\Theta^2(t)\nonumber\\&\qquad\qquad\qquad+\varepsilon^{-2}\Theta^p(t)
+\varepsilon^{-1}\Theta(t)\big\|\omega_{11}^A\big\|_{L^q(\Omega\times(0,t))}\Big)^{\big(\frac{1}{2}-\frac{\theta }{4}\big)p}
\nonumber\\&\quad +C\Big(\varepsilon^{-1}\Theta^p(t)+\Theta(t)\big\|\omega_{11}^A\big\|_{L^q(\Omega\times(0,t))}\Big)^\frac{ p}{2}
\label{recursive inequality}
\end{align}
where  $\omega_{11}^A=O(\varepsilon^{k-2})$.

Noting that if we  fix  $p\in (2,3]$, then
\begin{align*}
\frac{\theta p}{4}+2\big(\frac{1}{2}-\frac{\theta}{4}\big)=1+\frac{\theta(p-2)}{4}>1,\ \ \frac{p}{2}>1.
\end{align*}

According to the continuous argument, for small $\varepsilon$ and large $k$  there exits $\gamma=\gamma(k)\in(1,k)$ which tends to $+\infty$ as  $k\rightarrow+\infty$ such that
\begin{align*}
\Theta(T)\leq \varepsilon^{\gamma}.
\end{align*}

Hence
\begin{align*}
\big\|\varphi\big\|_{L^p(\Omega\times(0,T))}+\big\|\sigma\big\|_{L^p(\Omega\times(0,T))}\leq \varepsilon^{\gamma}
\end{align*}
and then
\begin{align*}
\big\|u\big\|_{L^p(\Omega\times(0,T))}&=\big\|\varphi-\sigma+\omega_5^A\big\|_{L^p(\Omega\times(0,T))}
\\&\leq\big\|\varphi\big\|_{L^p(\Omega\times(0,T))}+\big\|\sigma\big\|_{L^p(\Omega\times(0,T))}+\big\|\omega_5^A\big\|_{L^p(\Omega\times(0,T))}
\\[3pt]
&\leq C\varepsilon^{\gamma}.
\end{align*}
Therefore the proof of the theorem is completed.
\end{proof}

In order to establish   higher order regularity estimates of $u^{err},\varphi^{err}$ and $\sigma^{err}$, we consider the equations for $(u^{err},\sigma^{err})$
\begin{eqnarray*}
\left \{
\begin {array}{ll}
\partial_tu^{err}-\Delta \mu^{err}+\mu^{err}=2\sigma^{err}+u^{err}+\partial_t\omega_5^A,\quad&\text{in}\ \ \Omega\times(0,T),\\[5pt]
\partial_t\sigma^{err}-\Delta \sigma^{err}+2\sigma^{err}=\mu^{err}-u^{err},\ &\text{in}\ \ \Omega\times(0,T),
\end{array}
\right.
\end{eqnarray*}
which is a Cahn-Hilliard equation coupled linearly with a heat equation. Using the similar arguments in Theorem 2.3 in \cite{ABC} and the
boot-strap method we can give the desired conclusions. Here we omit the detailed argument. In particular we have

\begin{Theorem} \label{err:higher order}For small $\varepsilon$ and large $k$,
\begin{align*}
\big\|u^{err}\big\|_{C^{4,1}(\overline{\Omega}\times[0.T])}+\big\|\varphi^{err}\big\|_{C^{2,1}(\overline{\Omega}\times[0.T])}+\big\|\sigma^{err}\big\|_{C^{2,1}(\overline{\Omega}\times[0.T])}+\big\|\mu^{err}\big\|_{C^{2,1}(\overline{\Omega}\times[0.T])}\leq C\varepsilon.
\end{align*}

\end{Theorem}

Accordingly Theorem \ref{main theorem} can be obtained with the aid of Theorem \ref{err:higher order} if we take (\ref{initial data-impose}). Next we prove Corollary \ref{main corollary}.

\begin{proof}
Recalling the definition of $u^A$ in Section 5, we easily obtain (\ref{main corollary-1}). Now we prove (\ref{main corollary-2}). More concretely we only prove
\begin{align*}
\big\|\sigma^\varepsilon-\sigma\big\|_{C(\overline{\Omega}\times[0.T])}\rightarrow 0,\ \ \text{as}\ \  \varepsilon\rightarrow0.
\end{align*}
The other one in (\ref{main corollary-2}) is similar.

The definition of $\sigma^A$ in Section 5 yields
\begin{eqnarray*}
\text{the leading order of }\sigma^A=\left \{
\begin {array}{ll}
\sigma_+^{(0)},\ \qquad&\text{in}\ \ \Omega_+\backslash \Gamma(\delta),\\
\widetilde{\sigma}^{(0)}\zeta(\frac{d^{(0)}}{\delta})+\sigma_+^{(0)}\big(1-\zeta(\frac{d^{(0)}}{\delta})\big),\qquad&\text{in}\ \ \big(\Gamma(\delta)\backslash\Gamma(\frac{\delta}{2})\big)\cap\Omega_+,\\
\widetilde{\sigma}^{(0)},\ &\text{in}\ \ \Gamma(\frac{\delta}{2}),\\
\widetilde{\sigma}^{(0)}\zeta(\frac{d^{(0)}}{\delta})+\sigma_-^{(0)}\big(1-\zeta(\frac{d^{(0)}}{\delta})\big),\ &\text{in}\ \ \big(\Gamma(\delta)\backslash\Gamma(\frac{\delta}{2})\big)\cap\Omega_-,\\
\sigma_-^{(0)},\ &\text{in}\ \ \Omega_-\backslash \Gamma(\delta).
\end{array}
\right.
\end{eqnarray*}

Based on the inner-outer matching condition we find
\begin{align*}
\bigg\|\bigg(\widetilde{\sigma}^{(0)}\zeta(\frac{d^{(0)}}{\delta})+\sigma_{\pm}^{(0)}\big(1-\zeta(\frac{d^{(0)}}{\delta})\big)\bigg)-\sigma_{\pm}^{(0)}\bigg\|_{C\big(\big(\Gamma(\delta)\backslash\Gamma(\frac{\delta}{2})\big)\cap\Omega_\pm\big)}
\leq Ce^{-\frac{\delta}{4\varepsilon}}\rightarrow0,\ \ \text{as}\ \  \varepsilon\rightarrow0.
\end{align*}
Consequently we only need to prove
\begin{align*}
\big\|\widetilde{\sigma}^{(0)}-\sigma\big\|_{C(\Gamma(\frac{\delta}{2}))}\rightarrow 0,\ \ \text{as}\ \  \varepsilon\rightarrow0.
\end{align*}

Note that
\begin{eqnarray*}
\widetilde{\sigma}^{(0)}(x,t,z)=\left \{
\begin {array}{ll}
\sigma_+^{(0)}(x,t),\ \qquad&\text{in}\  \{(x,t):d^{[k]}(x,t)\geq\varepsilon\},\\[3pt]
\eta(z)\sigma_{+}^{(0)}(x,t)+\big(1-\eta(z)\big)\sigma_{-}^{(0)}(x,t),\ \ &\text{in}\  \{(x,t):-\varepsilon<d^{[k]}(x,t)<\varepsilon\},\\[3pt]
\sigma_-^{(0)}(x,t),\ &\text{in}\ \{(x,t):d^{[k]}(x,t)\leq-\varepsilon\},
\end{array}
\right.
\end{eqnarray*}
and
\begin{eqnarray*}
\sigma(x,t)=\left \{
\begin {array}{ll}
\sigma_+^{(0)}(x,t),\ \qquad&\text{in}\ \ \Omega_+,\\[3pt]
\kappa\int_{-1}^{1}\sqrt{2f(u)}du,\qquad &\text{on}\ \ \Gamma,\\[3pt]
\sigma_-^{(0)}(x,t),\ &\text{in}\ \ \Omega_-.
\end{array}
\right.
\end{eqnarray*}

By mean value theorem we derive  in $\Gamma(\frac{\delta}{2})$
\begin{align*}
\bigg|\sigma_+^{(0)}(x,t)-\kappa\int_{-1}^{1}\sqrt{2f(u)}du\bigg|,\ \bigg|\sigma_-^{(0)}(x,t)-\kappa\int_{-1}^{1}\sqrt{2f(u)}du\bigg|\leq C\big|d^{(0)}(x,t)\big|.
\end{align*}
Moreover,
\begin{align*}
\{(x,t):d^{[k]}(x,t)\geq\varepsilon\}\cap\Omega_-&\subseteq\{(x,t):|d^{(0)}(x,t)|\leq C\varepsilon\},\\
\{(x,t):-\varepsilon<d^{[k]}(x,t)<\varepsilon\}&\subseteq\{(x,t):|d^{(0)}(x,t)|\leq C\varepsilon\},\\
\{(x,t):d^{[k]}(x,t)\leq-\varepsilon\}\cap\Omega_+&\subseteq\{(x,t):|d^{(0)}(x,t)|\leq C\varepsilon\}.
\end{align*}
Then there holds
\begin{align*}
\big\|\widetilde{\sigma}^{(0)}-\sigma\big\|_{C(\Gamma(\frac{\delta}{2}))}\leq C\varepsilon,
\end{align*}
which implies the desired result.

The proof of Corollary \ref{main corollary} is completed.
\end{proof}
\indent

\section{Matching $\varepsilon^k(k\geq2)$-order expansions}\label{Matching k-order expansions}
\subsection{Matching kth order($k\geq2$)  outer expansion}\label{section:inexpan-k}
\indent

Substituting (\ref{out-1})-(\ref{out-3}) into (\ref{model}) and  collecting all terms of  $\varepsilon^k$-order($k\geq2$) we have
\begin{align}
u^{(k)}_{\pm}&=\frac{\mu_{\pm}^{(k-1)}+\Delta u^{(k-2)}_{\pm}-g\big(u^{(0)}_{\pm},\cdots,u^{(k-1)}_{\pm}\big)}{f''(u^{(0)}_{\pm})},\label{outerk-3}\\
-\Delta \mu^{(k)}_{\pm}+\mu^{(k)}_{\pm}&=2\sigma^{(k)}_{\pm}+u^{(k)}_{\pm}-\partial_tu^{(k)}_{\pm},\label{outerk-1}\\
\partial_t\sigma^{(k)}_{\pm}-\Delta \sigma^{(k)}_{\pm}+2\sigma^{(k)}_{\pm}&=\mu^{(k)}_{\pm}-u^{(k)}_{\pm}.\label{outerk-2}
\end{align}

\subsection{Matching kth order($k\geq2$)  inner expansion}\label{section:inexpan-k}

\indent

Substituting (\ref{d}) and (\ref{in-1})-(\ref{in-3}) into (\ref{inequ-m-1})-(\ref{inequ-m-3})  and collecting all terms of  $\varepsilon^k$-order  we have
\begin{align}
&-\partial_{zz}\Big(\widetilde{\mu}^{(k)}-\eta\big(p^{(k)}d^{(0)}+p^{(0)}d^{(k)}\big)\Big)
\nonumber\\&\qquad=-\sum_{i=0}^{k-1}\Big(\partial_z\widetilde{u}^{(i)}\partial_td^{(k-1-i)}-2\nabla_x\partial_z\widetilde{\mu}^{(i)}\cdot\nabla_x d^{(k-1-i)}-\partial_z\widetilde{\mu}^{(i)}\triangle_x d^{(k-1-i)} \Big)\nonumber
\\&\qquad\ \  -\Big(\partial_t\widetilde{u}^{(k-2)}-\Delta_x \widetilde{\mu}^{(k-2)}-\big(2\widetilde{\sigma}^{(k-2)}+\widetilde{u}^{(k-2)}-\widetilde{\mu}^{(k-2)}\big)\Big)
\nonumber\\&\qquad \ \ -\eta'' \sum_{i=1}^{k-1}p^{(i)}d^{(k-i)}+\eta''zp^{(k-1)}-\eta'\sum_{i=0}^{k-1} g^{(i)}d^{(k-1-i)}+z\eta'g^{(k-2)}
\nonumber\\&\qquad \ \ +\big(s_+^{(k-2)}\eta^++s_-^{(k-2)}\eta^-\big)
\nonumber\\&\qquad\triangleq \Theta_{k-1,1},\label{in-k-1}
\end{align}
\begin{align}
&-\partial_{zz}\Big(\widetilde{\sigma}^{(k)}-\eta\big(q^{(k)}d^{(0)}+q^{(0)}d^{(k)}\big)\Big)
\nonumber\\&\qquad=-\sum_{i=0}^{k-1}\Big(\partial_z\widetilde{\sigma}^{(i)}\partial_td^{(k-1-i)}-2\nabla_x\partial_z\widetilde{\sigma}^{(i)}\cdot\nabla_x d^{(k-1-i)}-\partial_z\widetilde{\sigma}^{(i)}\triangle_x d^{(k-1-i)} \Big)
\nonumber\\ &\qquad\ \ -\Big(\partial_t\widetilde{\sigma}^{(k-2)}-\Delta_x \widetilde{\sigma}^{(k-2)}+2\widetilde{\sigma}^{(k-2)}+\widetilde{u}^{(k-2)}-\widetilde{\mu}^{(k-2)}\Big)
\nonumber\\&\qquad \ \ -\eta'' \sum_{i=1}^{k-1}q^{(i)}d^{(k-i)}+\eta''zq^{(k-1)}- \eta'\sum_{i=0}^{k-1} h^{(i)}d^{(k-1-i)}+z\eta'h^{(k-2)}
\nonumber\\&\qquad\  \ +\big(r_+^{(k-2)}\eta^++r_-^{(k-2)}\eta^-\big)
\nonumber\\&\qquad\triangleq \Theta_{k-1,2},\label{in-k-2}\\
&-\partial_{zz}\widetilde{u}^{(k)}+f''(\widetilde{u}^{(0)})\widetilde{u}^{(k)}
=-\widetilde{g}\big(\widetilde{u}^{(0)},\cdots,\widetilde{u}^{(k-1)}\big)+2\sum_{i=0}^{k-1}\nabla_x\partial_z\widetilde{u}^{(i)}\cdot\nabla_x d^{(k-1-i)}
\nonumber\\&\qquad\qquad\qquad\qquad\qquad \qquad+\sum_{i=0}^{k-1}\partial_z\widetilde{u}^{(i)}\triangle_x d^{(k-1-i)}+\widetilde{\mu}^{(k-1)}+\triangle_x\widetilde{u}^{(k-2)}
\nonumber\\&\qquad\qquad\qquad\qquad\qquad \qquad-\eta'\sum_{i=0}^{k-1} l^{(i)}d^{(k-1-i)}+z\eta'l^{(k-2)}
\nonumber\\&\qquad\qquad\qquad\qquad\qquad \triangleq \Theta_{k-1,3}.\label{in-k-3}
\end{align}

\underline{Step 1.}
By induction we assume that the inner-outer matching conditions (\ref{mc-1})-(\ref{mc-3}) hold   for order up to $k-1$.
It follows from Lemma 4.3 in \cite{ABC} and direct computations that if

\begin{align}
\int_{-\infty}^{+\infty}\Theta_{k-1,1}(x,t,z)dz=0,\label{in-k-solving condition-1}
\end{align}
then (\ref{in-k-1}) has a bounded solution
\begin{align}
\widetilde{\mu}^{(k)}(x,t,z)=&\eta(z)\mu_+^{(k)}(x,t)+\big(1-\eta(z)\big)\mu_-^{(k)}(x,t)\nonumber\\&
-\eta(z)\int_{-\infty}^{+\infty}\int_{z'}^{+\infty}\Theta_{k-1,1}(x,t,z'')dz''dz'
\nonumber\\&+\int_{-\infty}^{z}\int_{z'}^{+\infty}\Theta_{k-1,1}(x,t,z'')dz''dz'\label{definition:muk}
\end{align}
which satisfies
\begin{align}
\big(p^{(k)}d^{(0)}+p^{(0)}d^{(k)}\big)(x,t)=\mu_+^{(k)}(x,t)-\int_{-\infty}^{+\infty}\int_{z'}^{+\infty}\Theta_{k-1,1}(x,t,z'')dz''dz'-\mu_-^{(k)}(x,t),\label{definition:pk-0}
\end{align}
and
\begin{align}
\big[\mu^{(k)}\big]\triangleq\mu_{+}^{(k)}-\mu_{-}^{(k)}=p^{(0)}d^{(k)}+\int_{-\infty}^{+\infty}\int_{z'}^{+\infty}\Theta_{k-1,1}(x,t,z'')dz''dz',\ \  \ \text{on}\ \ \Gamma,\label{jump condition:muk}
\end{align}
and for any $\alpha,\beta,\gamma\in\mathbb{N}$,
\begin{align*}
D_x^\alpha D_t^\beta D_z^\gamma\Big(\widetilde{\mu}^{(k)}(x,t,z) -\mu^{(k)}_{\pm}(x,t)\Big)=O(e^{-\nu|z|}),\ \ \text{ for} \ |z|\gg 1 \text{ and \ some } \nu>0.
\end{align*}

From (\ref{definition:pk-0})  we can take
\begin{eqnarray}
p^{(k)}=\left \{
\begin {array}{ll}
\frac{1}{d^{(0)}}\Big(\mu_+^{(k)}(x,t)-\mu_-^{(k)}(x,t)-\int_{-\infty}^{+\infty}\int_{z'}^{+\infty}\Theta_{k-1,1}(x,t,z'')dz''dz'
\\ \qquad\qquad\qquad\qquad\qquad\qquad-p^{(0)}d^{(k)}\Big),\quad&\text{in}\ \ \Gamma(\delta)\backslash\Gamma,\\
\\
\nabla_x d^{(0)}\cdot\nabla_x\Big(\mu_+^{(k)}(x,t)-\mu_-^{(k)}(x,t)-\int_{-\infty}^{+\infty}\int_{z'}^{+\infty}\Theta_{k-1,1}(x,t,z'')dz''dz'
\\ \qquad\qquad\qquad\qquad\qquad\qquad-p^{(0)}d^{(k)}\Big),\quad &\text{on} \ \ \Gamma.
\end{array}\label{definition:pk}
\right.
\end{eqnarray}

Moreover, by (\ref{in-k-solving condition-1}) we arrive at
\begin{align}
&2\partial_td^{(k-1)}=-g^{(0)}d^{(k-1)}-g^{(k-1)}d^{(0)}+2\nabla_x\big(\mu_+^{(0)}-\mu_-^{(0)}\big)\cdot\nabla_x d^{(k-1)}-\big(u_+^{(k-1)}-u_-^{(k-1)}\big)\partial_t d^{(0)}
\nonumber\\&\qquad\qquad\quad+2\nabla_x\big(\mu_+^{(k-1)}-\mu_-^{(k-1)}\big)\cdot\nabla_x d^{(0)}+\big(\mu_+^{(k-1)}-\mu_-^{(k-1)}\big)\Delta_x d^{(0)}\nonumber
\\&\qquad\qquad\quad -\sum_{i=1}^{k-2}\Big(\partial_z\widetilde{u}^{(i)}\partial_td^{(k-1-i)}-2\nabla_x\partial_z\widetilde{\mu}^{(i)}\cdot\nabla_x d^{(k-1-i)}-\partial_z\widetilde{\mu}^{(i)}\Delta_x d^{(k-1-i)} \Big)
\nonumber\\&\qquad\qquad\quad -p^{(k-1)}-\sum_{i=1}^{k-2} g^{(i)}d^{(k-1-i)}+g^{(k-2)}\int_{-\infty}^{+\infty}z\eta'dz
+\int_{-\infty}^{+\infty}\big(s_+^{(k-2)}\eta^++s_-^{(k-2)}\eta^-\big)dz
\nonumber\\&\qquad\qquad\quad -\int_{-\infty}^{+\infty}\big(\partial_t\widetilde{u}^{(k-2)}-\Delta_x \widetilde{\mu}^{(k-2)}-(2\widetilde{\sigma}^{(k-2)}+\widetilde{u}^{(k-2)}-\widetilde{\mu}^{(k-2)})\big)dz.\label{definition:gk1-0}
\end{align}

In particular, for $(x,t)\in\Gamma$ there holds
\begin{align}
&2\partial_td^{(k-1)}=-g^{(0)}d^{(k-1)}+2\big[\nabla_x\mu^{(0)}\big]\cdot\nabla_x d^{(k-1)}-\big[u^{(k-1)}\big]\partial_t d^{(0)}
\nonumber\\&\qquad\qquad\quad+2\bigg[\frac{\partial\mu^{(k-1)}}{\partial\nn}\bigg]+\big[\mu^{(k-1)}\big]\Delta_x d^{(0)}\nonumber
\\&\qquad\qquad\quad -\sum_{i=1}^{k-2}\Big(\partial_z\widetilde{u}^{(i)}\partial_td^{(k-1-i)}-2\nabla_x\partial_z\widetilde{\mu}^{(i)}\cdot\nabla_x d^{(k-1-i)}-\partial_z\widetilde{\mu}^{(i)}\Delta_x d^{(k-1-i)} \Big)
\nonumber\\&\qquad\qquad\quad -p^{(k-1)}-\sum_{i=1}^{k-2} g^{(i)}d^{(k-1-i)}+g^{(k-2)}\int_{-\infty}^{+\infty}z\eta'dz
\nonumber\\&\qquad\qquad\quad +\int_{-\infty}^{+\infty}\big(s_+^{(k-2)}\eta^++s_-^{(k-2)}\eta^-\big)dz
\nonumber\\&\qquad\qquad\quad -\int_{-\infty}^{+\infty}\big(\partial_t\widetilde{u}^{(k-2)}-\Delta_x \widetilde{\mu}^{(k-2)}-(2\widetilde{\sigma}^{(k-2)}+\widetilde{u}^{(k-2)}-\widetilde{\mu}^{(k-2)})\big)dz.\label{interface law:k-1-0}
\end{align}

Due to (\ref{jump condition:mu0}) there holds $\big[\nabla_x\mu^{(0)}\big]$  is parallel to $\nabla_x d^{(0)}$ on $\Gamma$ and then $\big[\nabla_x\mu^{(0)}\big]=p^{(0)}\nabla_x d^{(0)}$ on $\Gamma$ by (\ref{definition:p0}),
thus  for $(x,t)\in\Gamma$
\begin{eqnarray}
\big[\nabla_x\mu^{(0)}\big]\cdot\nabla_x d^{(k-1)}=\left \{
\begin {array}{ll}
0,\quad&k=2,\\
\\
-\frac{p^{(0)}}{2}\sum\limits_{i=1}^{k-2}\nabla_x d^{(i)}\cdot\nabla_x d^{(k-1-i)},\quad&k\geq3.
\end{array}\label{interface law:k-1-1}
\right.
\end{eqnarray}

Combining (\ref{interface law:k-1-1}), (\ref{outerk-3})$(k\rightarrow k-1)$,  (\ref{jump condition:muk})$(k\rightarrow k-1)$, (\ref{definition:pk})$(k\rightarrow k-1)$ with (\ref{interface law:k-1-0})
we obtain
\begin{align}
\partial_td^{(k-1)}=\frac{1}{2}\Big(p^{(0)}\Delta_x d^{(0)}-\nabla_x d^{(0)}\cdot\nabla_x p^{(0)}-g^{(0)}\Big)d^{(k-1)}+\frac{1}{2}\bigg[\frac{\partial\mu^{(k-1)}}{\partial\nn}\bigg]+\Lambda_{k-2,1},\ \ \text{on}\ \ \Gamma,\label{interface law:k-1}
\end{align}
where the function $\Lambda_{k-2,1}$ depends only on terms up to order $k-2$.

And by (\ref{definition:gk1-0}) one has
\begin{eqnarray}
g^{(k-1)}=\left \{
\begin {array}{ll}
\frac{1}{d^{(0)}}\Big(-2\partial_td^{(k-1)}-g^{(0)}d^{(k-1)}+2\nabla_x\big(\mu_+^{(0)}-\mu_-^{(0)}\big)\cdot\nabla_x d^{(k-1)}
\\ \qquad\quad-\big(u_+^{(k-1)}-u_-^{(k-1)}\big)\partial_t d^{(0)}
+2\nabla_x\big(\mu_+^{(k-1)}-\mu_-^{(k-1)}\big)\cdot\nabla_x d^{(0)}\\ \qquad\quad +\big(\mu_+^{(k-1)}-\mu_-^{(k-1)}\big)\Delta_x d^{(0)} -p^{(k-1)}+\Lambda_{k-2,2}\Big),&\text{in}\ \ \Gamma(\delta)\backslash\Gamma,\\
\\
\nabla d^{(0)}\cdot\nabla\Big(-2\partial_td^{(k-1)}-g^{(0)}d^{(k-1)}+2\nabla_x\big(\mu_+^{(0)}-\mu_-^{(0)}\big)\cdot\nabla_x d^{(k-1)}
\\ \qquad\qquad\quad-\big(u_+^{(k-1)}-u_-^{(k-1)}\big)\partial_t d^{(0)}
+2\nabla_x\big(\mu_+^{(k-1)}-\mu_-^{(k-1)}\big)\cdot\nabla_x d^{(0)}\\ \qquad\qquad\quad +\big(\mu_+^{(k-1)}-\mu_-^{(k-1)}\big)\Delta_x d^{(0)} -p^{(k-1)}+\Lambda_{k-2,2}\Big),\quad &\text{on} \ \ \Gamma,
\end{array}\label{definition:gk}
\right.
\end{eqnarray}
where the function $\Lambda_{k-2,2}$ depends only on terms up to order $k-2$.
%

Similarly,  if
\begin{align}
\int_{-\infty}^{+\infty}\Theta_{k-1,2}(x,t,z)dz=0,\label{in-k-solving condition-2}
\end{align}
then (\ref{in-k-2}) has a bounded solution
\begin{align*}
\widetilde{\sigma}^{(k)}(x,t,z)=&\eta(z)\sigma_+^{(k)}(x,t)+\big(1-\eta(z)\big)\sigma_-^{(k)}(x,t)\nonumber\\&
-\eta(z)\int_{-\infty}^{+\infty}\int_{z'}^{+\infty}\Theta_{k-1,2}(z'',x,t)dz''dz'
\nonumber\\&+\int_{-\infty}^{z}\int_{z'}^{+\infty}\Theta_{k-1,2}(x,t,z'')dz''dz'
\end{align*}
which satisfies
\begin{align*}
\big[\sigma^{(k)}\big]\triangleq\sigma_{+}^{(k)}-\sigma_{-}^{(k)}=q^{(0)}d^{(k)}+\int_{-\infty}^{+\infty}\int_{z'}^{+\infty}\Theta_{k-1,2}(x,t,z'')dz''dz',\ \  \ \text{on}\ \ \Gamma,
\end{align*}
and for any $\alpha,\beta,\gamma\in\mathbb{N}$,
\begin{align*}
D_x^\alpha D_t^\beta D_z^\gamma\Big(\widetilde{\sigma}^{(k)}(x,t,z) -\sigma^{(k)}_{\pm}(x,t)\Big)=O(e^{-\nu|z|}),\ \ \text{ for} \ |z|\gg 1 \text{ and \ some } \nu>0.
\end{align*}

Furthermore we can obtain
\begin{eqnarray}
q^{(k)}=\left \{
\begin {array}{ll}
\frac{1}{d^{(0)}}\Big(\sigma_+^{(k)}(x,t)-\sigma_-^{(k)}(x,t)-\int_{-\infty}^{+\infty}\int_{z'}^{+\infty}\Theta_{k-1,2}(x,t,z'')dz''dz'
\\ \qquad\qquad\qquad\qquad\qquad\qquad-q^{(0)}d^{(k)}\Big),\quad&\text{in}\ \ \Gamma(\delta)\backslash\Gamma,\\
\\
\nabla_x d^{(0)}\cdot\nabla_x\Big(\sigma_+^{(k)}(x,t)-\sigma_-^{(k)}(x,t)-\int_{-\infty}^{+\infty}\int_{z'}^{+\infty}\Theta_{k-1,2}(x,t,z'')dz''dz'
\\ \qquad\qquad\qquad\qquad\qquad\qquad-q^{(0)}d^{(k)}\Big),\quad &\text{on} \ \ \Gamma,
\end{array}\label{definition:qk}
\right.
\end{eqnarray}
and
\begin{align}
\bigg[\frac{\partial\sigma^{(k-1)}}{\partial\nn}\bigg]=\Big(\nabla_x d^{(0)}\cdot\nabla_x q^{(0)}+h^{(0)}-q^{(0)}\Delta_x d^{(0)}\Big)d^{(k-1)}+\Lambda_{k-2,3},\ \ \text{on}\ \ \Gamma,\label{interface law:sigma-k-1}
\end{align}
and
\begin{eqnarray}
h^{(k-1)}=\left \{
\begin {array}{ll}
\frac{1}{d^{(0)}}\Big(-h^{(0)}d^{(k-1)}-\big(\sigma_+^{(k-1)}-\sigma_-^{(k-1)}\big)\partial_t d^{(0)}
+2\nabla_x\big(\sigma_+^{(k-1)}-\sigma_-^{(k-1)}\big)\cdot\nabla_x d^{(0)}\\ \qquad\quad+\big(\sigma_+^{(k-1)}-\sigma_-^{(k-1)}\big)\Delta_x d^{(0)} -q^{(k-1)}+\Lambda_{k-2,4}\Big),\quad \text{in} \quad  \Gamma(\delta)\setminus \Gamma\\
\\
\nabla_x d^{(0)}\cdot\nabla_x\Big(-h^{(0)}d^{(k-1)}-\big(\sigma_+^{(k-1)}-\sigma_-^{(k-1)}\big)\partial_t d^{(0)}
\\ \qquad\qquad\qquad+2\nabla_x\big(\sigma_+^{(k-1)}-\sigma_-^{(k-1)}\big)\cdot\nabla_x d^{(0)}\\ \qquad\qquad\qquad +\big(\sigma_+^{(k-1)}-\sigma_-^{(k-1)}\big)\Delta_x d^{(0)} -q^{(k-1)}+\Lambda_{k-2,4}\Big),\quad  \text{on} \ \ \Gamma ,
\end{array}\label{definition:hk}
\right.
\end{eqnarray}
where the functions $\Lambda_{k-2,3}$ and $\Lambda_{k-2,4}$ depend only on terms up to order $k-2$.

\underline{Step 2.}
Based on the method of variation of constants and direct computations (or Lemma 4.3 in \cite{ABC}) we find that if
 \begin{align}
\int_{-\infty}^{+\infty}\Theta_{k-1,3}(x,t,z)\theta'(z)dz=0,\label{in-k-solving condition-3}
\end{align}
then (\ref{in-k-3}) has a bounded solution $\widetilde{u}^{(k)}(x,t,z)$ satisfying  $\widetilde{u}^{(k)}(0,x,t)=0$ and for any $\alpha,\beta,\gamma\in\mathbb{N}$,
\begin{align*}
D_x^\alpha D_t^\beta D_z^\gamma\Big(\widetilde{u}^{(k)}(x,t,z) -u^{(k)}_{\pm}(x,t)\Big)=O(e^{-\nu|z|}),\ \ \text{ for} \ |z|\gg 1 \text{ and \ some } \nu>0.
\end{align*}

According to (\ref{definition:muk}) and (\ref{in-k-solving condition-3}) we get
\begin{align}
&\mu_{+}^{(k-1)}\int_{-\infty}^{+\infty}\eta(z)\theta'(z)dz+\mu_{-}^{(k-1)}\int_{-\infty}^{+\infty}\big(1-\eta(z)\big)\theta'(z)dz
\nonumber\\&=-\Delta_x d^{(k-1)}\int_{-\infty}^{+\infty}\big(\theta'(z)\big)^2dz+l^{(0)}d^{(k-1)}\int_{-\infty}^{+\infty}\eta'(z)\theta'(z)dz
\nonumber\\&\quad+l^{(k-1)}d^{(0)}\int_{-\infty}^{+\infty}\eta'(z)\theta'(z)dz+\Lambda_{k-2,5},\label{interface condition:muk1-00}
\end{align}
where  the function $\Lambda_{k-2,5}$ depends only on terms up to order $k-2$.  Here we have used the fact $\widetilde{u}^{(k-1)}$ actually depends only on terms up to order $k-2$.
In particular, for $(x,t)\in\Gamma$ there holds
\begin{align}
&\mu_{+}^{(k-1)}\int_{-\infty}^{+\infty}\eta(z)\theta'(z)dz+\mu_{-}^{(k-1)}\int_{-\infty}^{+\infty}\big(1-\eta(z)\big)\theta'(z)dz
\nonumber\\&=-\Delta_x d^{(k-1)}\int_{-\infty}^{+\infty}\big(\theta'(z)\big)^2dz+l^{(0)}d^{(k-1)}\int_{-\infty}^{+\infty}\eta'(z)\theta'(z)dz+\Lambda_{k-2,5},\ \ \text{on}\ \ \Gamma.\label{interface condition:muk1-0}
\end{align}

It follows from (\ref{jump condition:muk})$(k\rightarrow k-1)$ and (\ref{interface condition:muk1-0}) that
\begin{align}
\mu_{\pm}^{(k-1)}&=-\frac{1}{2}\Delta_x d^{(k-1)}\int_{-\infty}^{+\infty}\big(\theta'(z)\big)^2dz+\frac{1}{2}l^{(0)}d^{(k-1)}\int_{-\infty}^{+\infty}\eta'(z)\theta'(z)dz
\nonumber\\&\quad+p^{(0)}d^{(k-1)}\int_{-\infty}^{+\infty}\Big(\frac{1}{2}-\eta(z)\pm\frac{1}{2}\Big)\theta'(z)dz+\Lambda_{k-2,6}^{\pm},\ \ \ \ \text{on}\ \ \Gamma,\label{interface condition:muk1}
\end{align}
where  the functions $\Lambda_{k-2,6}^{\pm}$ depend only on terms up to order $k-2$.

And by (\ref{interface condition:muk1-00}) one has
\begin{eqnarray}
l^{(k-1)}=\left \{
\begin {array}{ll}
\frac{1}{d^{(0)}\int_{-\infty}^{+\infty}\eta'(z)\theta'(z)dz}\Big(\mu_{+}^{(k-1)}\int_{-\infty}^{+\infty}\eta(z)\theta'(z)dz+\mu_{-}^{(k-1)}\int_{-\infty}^{+\infty}\big(1-\eta(z)\big)\theta'(z)dz
\\[3pt]
 \qquad\qquad\qquad\qquad\ +\Delta_x d^{(k-1)}\int_{-\infty}^{+\infty}\big(\theta'(z)\big)^2dz-l^{(0)}d^{(k-1)}\int_{-\infty}^{+\infty}\eta'(z)\theta'(z)dz
\\[3pt] \qquad\qquad\qquad\qquad\  +\Lambda_{k-2,7}\Big),\quad \text{in}\ \ \Gamma(\delta)\backslash\Gamma,\\
\\
\frac{1}{\int_{-\infty}^{+\infty}\eta'(z)\theta'(z)dz}\nabla_x d^{(0)}\cdot\nabla_x\Big(\mu_{+}^{(k-1)}\int_{-\infty}^{+\infty}\eta(z)\theta'(z)dz+\mu_{-}^{(k-1)}\int_{-\infty}^{+\infty}\big(1-\eta(z)\big)\theta'(z)dz
\\[3pt] \qquad\qquad\qquad\qquad\qquad\quad +\Delta_x d^{(k-1)}\int_{-\infty}^{+\infty}\big(\theta'(z)\big)^2dz-l^{(0)}d^{(k-1)}\int_{-\infty}^{+\infty}\eta'(z)\theta'(z)dz
\\[3pt] \qquad\qquad\qquad\qquad\qquad\quad  +\Lambda_{k-2,7}\Big),\quad  \text{on} \ \ \Gamma,
\end{array}\label{definition:lk}
\right.
\end{eqnarray}
where the function $\Lambda_{k-2,7}$ depends only on terms up to order $k-2$.
\subsection{Matching kth order($k\geq2$)  boundary layer expansion}\label{section:bdexpan-k}
\indent

Substituting (\ref{bd-1})-(\ref{bd-3}) into (\ref{bdequ-1})-(\ref{bdequ-3}) and (\ref{bdc-1})-(\ref{bdc-3}) and collecting all terms of  $\varepsilon^k$-order($k\geq2$) we have
\begin{align}
&-\partial_{zz}\mu_B^{(k)}=\Xi_{k-1,1},\label{bdequ-1k}\\
&-\partial_{zz}\sigma_B^{(k)}=\Xi_{k-1,2},\label{bdequ-2k}\\
&-\partial_{zz}u_B^{(k)}+f''(1)u_B^{(k)}=\Xi_{k-1,3},\label{bdequ-3k}
\end{align}
 and on $\partial\Omega\times[0,T]$
\begin{align}
&\partial_{z}\mu_B^{(0)}(x,t,0)=0,\ \ \partial_{z}\mu_B^{(k)}(x,t,0)=-\nabla_x\mu_B^{(k-1)}(x,t,0)\cdot \nabla_x d_B(x,t),\label{bdc-1k}\\
&\partial_{z}\sigma_B^{(0)}(x,t,0)=0,\ \ \partial_{z}\sigma_B^{(k)}(x,t,0)=-\nabla_x\sigma_B^{(k-1)}(x,t,0)\cdot \nabla_x d_B(x,t),\label{bdc-2k}\\
&\partial_{z}u_B^{(0)}(x,t,0)=0,\ \ \partial_{z}u_B^{(k)}(x,t,0)=-\nabla_xu_B^{(k-1)}(x,t,0)\cdot \nabla_x d_B(x,t),\label{bdc-3k}
\end{align}
where the functions $\Xi_{k-1,1},\Xi_{k-1,2}$ and $\Xi_{k-1,3}$ depend only on the terms up to order $k-1.$  More concretely,
\begin{align*}
&\Xi_{k-1,1}=2\nabla_x\partial_z\mu_B^{(k-1)}\cdot \nabla_x d_B+\partial_z\mu_B^{(k-1)} \Delta_x d_B-\partial_tu_B^{(k-2)}+\Delta_x \mu_B^{(k-2)}\\& \qquad\qquad\qquad+2\sigma_B^{(k-2)}+u_B^{(k-2)}-\mu_B^{(k-2)},\\
&\Xi_{k-1,2}=2\nabla_x\partial_z\sigma_B^{(k-1)}\cdot\nabla_x d_B+\partial_z\sigma_B^{(k-1)}\Delta_x d_B-\partial_t\sigma_B^{(k-2)}+\Delta_x \sigma_B^{(k-2)}
\\& \qquad\qquad\qquad-\big(2\sigma_B^{(k-2)}+u_B^{(k-2)}-\mu_B^{(k-2)}\big),\\
&\Xi_{k-1,3}=-g_{_B}\big(u_B^{(0)},\cdots,u_B^{(k-1)}\big)+2\nabla_x\partial_zu_B^{(k-1)}\cdot\nabla_x d_B+\partial_zu_B^{(k-1)}\Delta_x d_B
\\& \qquad\qquad\qquad+\mu_B^{(k-1)}+\Delta_x u_B^{(k-2)}.
\end{align*}

By induction we assume that the boundary-outer matching conditions (\ref{bmc-1})-(\ref{bmc-3}) hold   for order up to $k-1$.
We then by  (\ref{bdequ-1k})-(\ref{bdequ-2k}) get
\begin{align}
&\mu_B^{(k)}(x,t,z)=-\int_{-\infty}^z\int_{-\infty}^{z'}\Xi_{k-1,1}(x,t,z'')dz''dz'+\mu_+^{(k)}(x,t),\label{def:bdequ-1k}\\
&\sigma_B^{(k)}(x,t,z)=-\int_{-\infty}^z\int_{-\infty}^{z'}\Xi_{k-1,2}(x,t,z'')dz''dz'+\sigma_+^{(k)}(x,t),\label{def:bdequ-2k}
\end{align}
and (\ref{bmc-2})-(\ref{bmc-3}) for order $k$ are satisfied by induction arguments.

In order that $\mu_B^{(k)}$ defined by (\ref{def:bdequ-1k}) satisfies  (\ref{bdc-1k})($k\geq2$) and $\sigma_B^{(k)}$ defined by (\ref{def:bdequ-2k}) satisfies (\ref{bdc-2k})($k\geq2$), we only need to assume on $\partial\Omega\times[0,T]$ that
\begin{align}
&\nabla_x d_B(x,t)\cdot\nabla_x\mu_+^{(k-1)}(x,t)
\nonumber\\&=-\Big(2\nabla_x d_B(x,t)\cdot\nabla_x+\Delta_x d_B(x,t)\Big)\int_{-\infty}^0\int_{-\infty}^{z'}\Xi_{k-2,1}(x,t,z'')dz''dz'
\nonumber\\&\quad-\int_{-\infty}^{0}\Big(\partial_tu_B^{(k-2)}-\Delta_x \mu_B^{(k-2)}-2\sigma_B^{(k-2)}-u_B^{(k-2)}+\mu_B^{(k-2)}\Big)(x,t,z)dz
\nonumber\\&\triangleq\Pi_{k-2,1}\label{bdc-1k-0}
\end{align}
and
\begin{align}
&\nabla_x d_B(x,t)\cdot\nabla_x\sigma_+^{(k-1)}(x,t)\nonumber\\&=-\Big(2\nabla_x d_B(x,t)\cdot\nabla_x+\Delta_x d_B(x,t)\Big)\int_{-\infty}^0\int_{-\infty}^{z'}\Xi_{k-2,2}(x,t,z'')dz''dz'
\nonumber\\&\quad-\int_{-\infty}^{0}\Big(\partial_t\sigma_B^{(k-2)}-\Delta_x \sigma_B^{(k-2)}-2\sigma_B^{(k-2)}-u_B^{(k-2)}+\mu_B^{(k-2)}\Big)(x,t,z)dz
\nonumber\\&\triangleq\Pi_{k-2,2}.\label{bdc-2k-0}
\end{align}

In fact, for $(x,t)\in\partial\Omega\times[0,T]$ one has
\begin{align*}
-\partial_{z}\mu_B^{(k)}(x,t,0)&=\int_{-\infty}^{0}\Xi_{k-1,1}(x,t,z)dz
\nonumber\\&=\Big(2\nabla_x d_B(x,t)\cdot\nabla_x+\Delta_x d_B(x,t)\Big)\mu_B^{(k-1)}(x,t,0)
\nonumber\\&\quad-\Big(2\nabla_x d_B(x,t)\cdot\nabla+\Delta_x d_B(x,t)\Big)\mu_+^{(k-1)}(x,t)
\nonumber\\&\quad-\int_{-\infty}^{0}\Big(\partial_tu_B^{(k-2)}-\Delta_x \mu_B^{(k-2)}-2\sigma_B^{(k-2)}-u_B^{(k-2)}+\mu_B^{(k-2)}\Big)(x,t,z)dz
\nonumber\\&=-\Big(2\nabla_x d_B(x,t)\cdot\nabla_x+\Delta_x d_B(x,t)\Big)\int_{-\infty}^0\int_{-\infty}^{z'}\Xi_{k-2,1}(x,t,z'')dz''dz'
\nonumber\\&\quad-\int_{-\infty}^{0}\Big(\partial_tu_B^{(k-2)}-\Delta_x \mu_B^{(k-2)}-2\sigma_B^{(k-2)}-u_B^{(k-2)}+\mu_B^{(k-2)}\Big)(x,t,z)dz
\end{align*}
and
\begin{align*}
\nabla_x\mu_B^{(k-1)}(x,t,0)\cdot \nabla_x d_B(x,t)&=-\nabla_x d_B(x,t)\cdot\nabla_x\bigg(\int_{-\infty}^0\int_{-\infty}^{z'}\Xi_{k-2,1}(x,t,z'')dz''dz'\bigg)
\nonumber\\&\quad+\nabla_x d_B(x,t)\cdot\nabla_x\mu_+^{(k-1)}(x,t).
\end{align*}
Then we easily get (\ref{bdc-1k})($k\geq2$) with the help of (\ref{bdc-1k-0}) and the above equalities. The other case is complete similar.

Finally, we equip (\ref{bdequ-3k})($k\geq2$) with the following boundary condition at $z=0$
\begin{align}
\partial_{z}u_B^{(k)}(x,t,0)=-\nabla_xu_B^{(k-1)}(x,t,0)\cdot \nabla_x d_B(x,t).\label{bdc-3k-0}
\end{align}
for $(x,t)\in\partial\Omega(\delta)\times[0,T]$. Obviously (\ref{bdc-3k-0}) implies (\ref{bdc-3k}).

Since (\ref{bdequ-3k})($k\geq1$) is a  linear second-order ordinary differential equation with constant coefficients, we can solve it explicitly and conclude that there exists a unique  solution $u_B^{(k)}$ which satisfies  (\ref{bdc-3k-0}) and (\ref{bmc-1}).
\subsection{Solving expansions of $k$th order($k\geq 1$)}
\indent

%
%
Assuming  $u_{\pm}^{(k-1)},\mu_{\pm}^{(k-1)},\sigma_{\pm}^{(k-1)}, d^{(k-1)}$, $p^{(k-1)},q^{(k-1)},g^{(k-1)},h^{(k-1)},l^{(k-1)}$, $\widetilde{u}^{(k-1)},\widetilde{\mu}^{(k-1)},\widetilde{\sigma}^{(k-1)}$,   $\mu^{(k-1)}_{B}, \sigma^{(k-1)}_{B},u^{(k-1)}_{B}$  are known and  the inner-outer matching conditions (\ref{mc-1})-(\ref{mc-3}), the boundary-outer matching conditions (\ref{bmc-1})-(\ref{bmc-3}) hold   for order up to $k-1$. Then $u^{(k)}_{\pm}$ are defined by (\ref{outerk-3}). Combining (\ref{outerk-1}), (\ref{outerk-2}), (\ref{equ:gradient d}), (\ref{interface condition:muk1})($k-1\rightarrow k$), (\ref{interface law:sigma-k-1})($k-1\rightarrow k$),  (\ref{interface law:k-1})($k-1\rightarrow k$), (\ref{bdc-1k-0})($k-1\rightarrow k$), (\ref{bdc-2k-0})($k-1\rightarrow k$), we have
\begin{eqnarray}
\left \{
\begin {array}{ll}
-\Delta \mu^{(k)}_{\pm}+\mu^{(k)}_{\pm}=2\sigma^{(k)}_{\pm}+u^{(k)}_{\pm}-\partial_tu^{(k)}_{\pm},\qquad&\text{in}\ \ \Omega_{\pm},\\[3pt]
\partial_t\sigma^{(k)}_{\pm}-\Delta \sigma^{(k)}_{\pm}+2\sigma^{(k)}_{\pm}=\mu^{(k)}_{\pm}-u^{(k)}_{\pm},\quad&\text{in}\ \ \Omega_{\pm},\\[3pt]
\nabla d^{(0)}\cdot\nabla d^{(k)}=\mathcal{D}_{k-1},&\text{in}\ \ \Gamma(\delta),\\[3pt]
\mu_{\pm}^{(k)}=-a_{0}\Delta d^{(k)}+a_{1}d^{(k)}+\Lambda_{k-1,6}^{\pm},\ \ \ \ &\text{on}\ \ \Gamma,\\[3pt]
\big[\frac{\partial\sigma^{(k)}}{\partial\nn}\big]=a_2d^{(k)}+\Lambda_{k-1,3},\ \ &\text{on}\ \ \Gamma,\\[3pt]
\partial_td^{(k)}=a_3d^{(k)}+\frac{1}{2}\big[\frac{\partial\mu^{(k)}}{\partial\nn}\big]+\Lambda_{k-1,1},\ \ &\text{on}\ \ \Gamma,\\[3pt]
\frac{\partial\mu_+^{(k)}}{\partial\nn}=\Pi_{k-1,1},\ \ &\text{on}\ \ \partial\Omega,\\[3pt]
 \frac{\partial\sigma_+^{(k)}}{\partial\nn}=\Pi_{k-1,2},\ \ &\text{on}\ \ \partial\Omega,\\[3pt]
 d^{(k)}(x,0)=0,&\text{on}\ \ \Gamma_0,\label{kmodel}
\end{array}
\right.
\end{eqnarray}
where $a_0$ is a positive constant, the functions $a_1, a_2$ and $a_3$ depend on $p^{(0)},q^{(0)},g^{(0)},h^{(0)},l^{(0)}$ and $d^{(0)}$, and $\Gamma_0=\Gamma|_{t=0}$.
Giving initial data $\sigma^{(k)}_{\pm}(x,0)$ and solving (\ref{kmodel}) leads to $\mu_{\pm}^{(k)},\sigma_{\pm}^{(k)}$ and $d^{(k)}$. (\ref{kmodel}) is a ``linearized'' Hele-Shaw problem(P193 in \cite{ABC}) coupled linearly with a heat equation satisfied by $\sigma^{(k)}_{\pm}$. The first and key strategy is to get the value of $d^{(k)}$ on $\Gamma$. Here we don't aim to show the lengthy details and one can refer to the similar arguments in Section 6 of \cite{ABC}.

Then $p^{(k)},q^{(k)},g^{(k)},h^{(k)},l^{(k)}$ are determined by (\ref{definition:pk}), (\ref{definition:qk}), (\ref{definition:gk}), (\ref{definition:hk}), (\ref{definition:lk}) respectively. Moreover $\widetilde{u}^{(k)},\widetilde{\mu}^{(k)},\widetilde{\sigma}^{(k)}$ are determined in section \ref{section:inexpan-k},   $\mu^{(k)}_{B}, \sigma^{(k)}_{B},u^{(k)}_{B}$ are determined in section  \ref{section:bdexpan-k}, and  the inner-outer matching conditions (\ref{mc-1})-(\ref{mc-3}) and  the boundary-outer matching conditions (\ref{bmc-1})-(\ref{bmc-3}) hold   for order $k$.
\begin{Remark}
We can extend $(u^{(k)}_{\pm}, \mu^{(k)}_{\pm}, \sigma^{(k)}_{\pm})$ smoothly from $\Omega_{\pm}$ to $\Omega$ as in Remark 4.1 in  \cite{ABC}.
\end{Remark}
\indent

\section{Construction of an approximate solution}\label{onstruction of an approximate solution}
\indent

In this section we divide into two steps to construct an approximate solution and determine the system which is satisfied by the approximate solution.

\underline{Step 1.}
In $\Omega_{+}\cup\Omega_{-}$ we define
\begin{align*}
u_O^A(x,t)=\bigg(\sum_{i=0}^k\varepsilon^iu^{(i)}_{+}(x,t)\bigg)\chi_{\Omega_+}(x,t)+\bigg(\sum_{i=0}^k\varepsilon^iu^{(i)}_{-}(x,t)\bigg)\chi_{\Omega_-}(x,t),\\
\mu_O^A(x,t)=\bigg(\sum_{i=0}^k\varepsilon^i\mu^{(i)}_{+}(x,t)\bigg)\chi_{\Omega_+}(x,t)+\bigg(\sum_{i=0}^k\varepsilon^i\mu^{(i)}_{-}(x,t)\bigg)\chi_{\Omega_-}(x,t),\\
\sigma_O^A(x,t)=\bigg(\sum_{i=0}^k\varepsilon^i\sigma^{(i)}_{+}(x,t)\bigg)\chi_{\Omega_+}(x,t)+\bigg(\sum_{i=0}^k\varepsilon^i\sigma^{(i)}_{-}(x,t)\bigg)
\chi_{\Omega_-}(x,t),
\end{align*}
where $\chi_{\Omega_{\pm}}$ is the characteristic function of $\Omega_{\pm}$.

Thanks to outer matching expansion procedure, we obtain in $\Omega_{+}\cup\Omega_{-}$
\begin{align*}
\partial_tu_O^A-\Delta \mu_O^A&=2\sigma_O^A+u_O^A-\mu_O^A,\\
\partial_t\sigma_O^A-\Delta \sigma_O^A&=-(2\sigma_O^A+u_O^A-\mu_O^A),\\
\mu_O^A&=-\varepsilon\Delta u_O^A+\varepsilon^{-1}f'(u_O^A)+O(\varepsilon^{k}).
\end{align*}

In $\Gamma(\delta)$ we define
\begin{align*}
u_I^A(x,t)=\sum_{i=0}^k\varepsilon^i\tilde{u}^{(i)}(x,t,z)\bigg|_{z=\frac{d^{[k]}(x,t)}{\varepsilon}},\\
\mu_I^A(x,t)=\sum_{i=0}^k\varepsilon^i\tilde{\mu}^{(i)}(x,t,z)\bigg|_{z=\frac{d^{[k]}(x,t)}{\varepsilon}},\\
\sigma_I^A(x,t)=\sum_{i=0}^k\varepsilon^i\tilde{\sigma}^{(i)}(x,t,z)\bigg|_{z=\frac{d^{[k]}(x,t)}{\varepsilon}}.
\end{align*}

Thanks to inner matching expansion procedure, we obtain in $\Gamma(\delta)$
\begin{align*}
\partial_tu_I^A-\Delta \mu_I^A&=2\sigma_I^A+u_I^A-\mu_I^A+O(\varepsilon^{k-1}),\\
\partial_t\sigma_I^A-\Delta \sigma_I^A&=-(2\sigma_I^A+u_I^A-\mu_I^A)+O(\varepsilon^{k-1}),\\
\mu_I^A&=-\varepsilon\Delta u_I^A+\varepsilon^{-1}f'(u_I^A)+O(\varepsilon^{k}).
\end{align*}

In $\partial\Omega(\delta)$ we define
\begin{align*}
u_B^A(x,t)=\sum_{i=0}^k\varepsilon^iu_B^{(i)}(x,t,z)\bigg|_{z=\frac{d_B(x)}{\varepsilon}}-\varepsilon^ku_B^{(k)}(0,x,t),\\
\mu_B^A(x,t)=\sum_{i=0}^k\varepsilon^i\mu_B^{(i)}(x,t,z)\bigg|_{z=\frac{d_B(x)}{\varepsilon}}-\varepsilon^k\mu_B^{(k)}(0,x,t),\\
\sigma_B^A(x,t)=\sum_{i=0}^k\varepsilon^i\sigma_B^{(i)}(x,t,z)\bigg|_{z=\frac{d_B(x)}{\varepsilon}}-\varepsilon^k\sigma_B^{(k)}(0,x,t).
\end{align*}

Thanks to boundary matching expansion procedure, we obtain in $\partial\Omega(\delta)$
\begin{align*}
\partial_tu_B^A-\Delta \mu_B^A&=2\sigma_B^A+u_B^A-\mu_B^A+O(\varepsilon^{k-1}),\\
\partial_t\sigma_B^A-\Delta \sigma_B^A&=-(2\sigma_B^A+u_B^A-\mu_B^A)+O(\varepsilon^{k-1}),\\
\mu_B^A&=-\varepsilon\Delta u_B^A+\varepsilon^{-1}f'(u_B^A)+O(\varepsilon^{k-1})
\end{align*}
and
\begin{align*}
\frac{\partial u_B^A}{\partial\nn}=\frac{\partial \mu_B^A}{\partial\nn}=\frac{\partial \sigma_B^A}{\partial\nn}=0,\ \ \ \text{on}\ \ \partial\Omega\times(0,T).
\end{align*}

\underline{Step 2.}
We define  $(\overline{u^A},\overline{\mu^A},\overline{\sigma^A})$ as follows:
\begin{eqnarray*}
\overline{u^A}=\left \{
\begin {array}{ll}
u_B^A,\ \qquad\qquad\qquad&\text{in}\ \ \partial\Omega(\frac{\delta}{2}),\\[3pt]
u_B^A\zeta(\frac{d_B}{\delta})+u_O^A\big(1-\zeta(\frac{d_B}{\delta})\big),\ \qquad\qquad\qquad&\text{in}\ \ \partial\Omega(\delta)\backslash\partial\Omega(\frac{\delta}{2}),\\[3pt]
u_O^A,\ \qquad\qquad\qquad&\text{in}\ \ \Omega\backslash\big(\partial\Omega(\delta)\cup\Gamma(\delta)\big),\\[3pt]
u_I^A\zeta(\frac{d^{(0)}}{\delta})+u_O^A\big(1-\zeta(\frac{d^{(0)}}{\delta})\big),\ \qquad\qquad\qquad&\text{in}\ \ \Gamma(\delta)\backslash\Gamma(\frac{\delta}{2}),\\[3pt]
u_I^A,\ \qquad\qquad\qquad&\text{in}\ \ \Gamma(\frac{\delta}{2}),
\end{array}
\right.
\end{eqnarray*}
and $\overline{\mu^A},\overline{\sigma^A}$ are defined similarly, where
\beno
\zeta\in C_c^\infty(R), \quad \zeta=1 \quad {\rm for}\quad |\zeta|\leq \frac12,\quad \zeta=0 \quad {\rm for}\quad |\zeta|\geq 1.
\eeno

Based on the boundary-outer matching conditions (\ref{bmc-1})-(\ref{bmc-3}) and  inner-outer matching conditions (\ref{mc-1})-(\ref{mc-3}), one has  for small $\varepsilon$
\begin{align*}
\big\|\overline{u^A}-u_O^A\big\|_{C^2(\partial\Omega(\delta)\backslash\partial\Omega(\frac{\delta}{2}))}&=\big\|(u_B^A-u_O^A)\zeta(\frac{d_B}{\delta})\big\|_{C^2(\partial\Omega(\delta)\backslash\partial\Omega(\frac{\delta}{2}))}
\\&=O(\varepsilon^2e^{-\frac{\nu\delta}{2\varepsilon}})+O(\varepsilon^k)
\end{align*}
and
\begin{align*}
\big\|\overline{u^A}-u_O^A\big\|_{C^2(\Gamma(\delta)\backslash\Gamma(\frac{\delta}{2}))}&=\big\|(u_I^A-u_O^A)\zeta(\frac{d^{(0)}}{\delta})\big\|_{C^2(\Gamma(\delta)\backslash\Gamma(\frac{\delta}{2}))}
\\&=O(\varepsilon^2e^{-\frac{\nu\delta}{4\varepsilon}}).
\end{align*}
And we can obtain the similar results for $\overline{\mu^A}-\mu_O^A$ and $\overline{\sigma^A}-\sigma_O^A$.

Consequently  $(\overline{u^A},\overline{\mu^A},\overline{\sigma^A})$ satisfies in $\Omega\times(0,T)$
\begin{align*}
\partial_t\overline{u^A}-\Delta \overline{\mu^A}&=2\overline{\sigma^A}+\overline{u^A}-\overline{\mu^A}+\omega^A_1,\\
\partial_t\overline{\sigma^A}-\Delta \overline{\sigma^A}&=-(2\overline{\sigma^A}+\overline{u^A}-\overline{\mu^A})+\omega^A_2,\\
\overline{\mu^A}&=-\varepsilon\Delta\overline{ u^A}+\varepsilon^{-1}f'(\overline{u^A})+\omega^A_3
\end{align*}
and
\begin{align*}
\frac{\partial \overline{u^A}}{\partial\nn}=\frac{\partial\overline{ \mu^A}}{\partial\nn}=\frac{\partial \overline{\sigma^A}}{\partial\nn}=0,\ \ \ \text{on}\ \ \partial\Omega\times(0,T),
\end{align*}
where $\omega^A_i=O(\varepsilon^{k-1})(i=1,2,3)$ which depend on $(\overline{u^A},\overline{\mu^A},\overline{\sigma^A}).$

Let $\overline{\varphi^A}=\overline{u^A}+\overline{\sigma^A}$, then $\overline{\varphi^A}$ satisfies
\begin{eqnarray*}
\left \{
\begin {array}{ll}
\partial_t\overline{\varphi^A}-\Delta \overline{\mu^A}-\Delta \overline{\sigma^A}=\omega^A_1+\omega^A_2,\ &\text{in}\ \ \Omega\times(0,T),\\
\frac{\partial \overline{\varphi^A}}{\partial\nn}=0,\quad &\text{on}\ \partial\Omega\times(0,T).
\end{array}
\right.
\end{eqnarray*}

Define the  approximate solution $(\varphi^A,\mu^A,\sigma^A)$ as follows:
\begin{eqnarray}
\left \{
\begin {array}{ll}
\varphi^A(x,t)=\overline{\varphi^A}(x,t)-\frac{1}{|\Omega|}\int_0^t\int_{\Omega}\big(\omega^A_1+\omega^A_2\big)(x,t')dxdt',\\[3pt]
\mu^A(x,t)=\overline{\mu^A}(x,t)-\tilde{\mu}^A(x,t),\\[3pt]
\sigma^A(x,t)=\overline{\sigma^A}(x,t),\\[3pt]
u^A(x,t)=\overline{u^A}(x,t)-\omega^A_2(x,t)-\tilde{\mu}^A(x,t),\label{approximate solutions}
\end{array}
\right.
\end{eqnarray}
where $\tilde{\mu}^A(x,t)$ satisfies
\begin{eqnarray*}
\left \{
\begin {array}{ll}
\Delta \tilde{\mu}^A=\omega^A_1+\omega^A_2-\frac{1}{|\Omega|}\int_{\Omega}\big(\omega^A_1+\omega^A_2\big)(x,t)dx,\ &\text{in}\ \ \Omega\times(0,T),\\
\frac{\partial \tilde{\mu}^A}{\partial\nn}=0,\quad &\text{on}\ \partial\Omega\times(0,T),\\[3pt]
\int_{\Omega}\tilde{\mu}^A(x,t)dx=0,\quad &t\in(0,T).
\end{array}
\right.
\end{eqnarray*}

Then the  approximate solution $(\varphi^A,\mu^A,\sigma^A)$ satisfies
\begin{eqnarray}
\left \{
\begin {array}{ll}
\partial_t\varphi^A-\Delta \mu^A-\Delta \sigma^A=0,\ &\text{in}\ \ \Omega\times(0,T),\\
\partial_t\sigma^A-\Delta \sigma^A=-(2\sigma^A+u^A-\mu^A),\ &\text{in}\ \ \Omega\times(0,T),\\
\mu^A=-\varepsilon\Delta u^A+\varepsilon^{-1}f'(u^A)+\omega_4^A,\ &\text{in}\ \ \Omega\times(0,T),\\
\frac{\partial \varphi^A}{\partial\nn}=\frac{\partial \mu^A}{\partial\nn}=\frac{\partial\sigma^A}{\partial\nn}=0,\ &\text{on}\ \ \partial\Omega\times(0,T),\label{approximate model}
\end{array}
\right.
\end{eqnarray}
where $\omega_4^A=O(\varepsilon^{k-2})$.
\indent

\section{Appendix}
\indent

In this Appendix we give the proof of (\ref{Jacobic transform}) which has been proved in \cite{FWZZ}. Here we repeat the proof out of the completeness of this paper.
\begin{proof}
Firstly we can observe that
\begin{align}
\int_{-\frac{\delta}{\ve}}^{\frac{\delta}{\ve}}\big((\partial_zv)^2+f''\big(\theta(z)\big)v^2\big)Jdz\geq&\int_{-\frac{\delta}{\ve}}^{\frac{\delta}{\ve}}
\big((\partial_z\hat{v})^2+f''\big(\theta(z)\big)\hat{v}^2\big)dz-\ve\int_{-\frac{\delta}{\ve}}^{\frac{\delta}{\ve}}\hat{v}\partial_z\hat{v}
J_{r_t}J^{-1}dz
\nonumber\\&-C\ve^2\int_{-\frac{\delta}{\ve}}^{\frac{\delta}{\ve}}v^2dz.\label{Jacobian inequality -3}
\end{align}
Let $\hat{v}=\gamma q_1^f+p_1$, then
\begin{align}
-\ve\int_{-\frac{\delta}{\ve}}^{\frac{\delta}{\ve}}\hat{v}\partial_z\hat{v}J_{r_t}J^{-1}dz&=-\ve\gamma
\int_{-\frac{\delta}{\ve}}^{\frac{\delta}{\ve}}\hat{v}(q_1^f)'J_{r_t}J^{-1}dz-\ve\int_{-\frac{\delta}{\ve}}^{\frac{\delta}{\ve}}
\hat{v}\partial_z{p_1}J_{r_t}J^{-1}dz
\nonumber\\&=-\ve\gamma\int_{-\frac{\delta}{\ve}}^{\frac{\delta}{\ve}}\hat{v}(q_1^f-\alpha \theta')'J_{r_t}J^{-1}dz+\ve\alpha\gamma\int_{-\frac{\delta}{\ve}}^{\frac{\delta}{\ve}}\hat{v}\theta''J_{r_t}J^{-1}dz
\nonumber\\&\quad-\ve\int_{-\frac{\delta}{\ve}}^{\frac{\delta}{\ve}}\hat{v}\partial_z{p_1}J_{r_t}J^{-1}dz.\label{Jacobian inequality-4}
\end{align}

We easily find
\begin{align}
-\big(q^f_1-\alpha \theta'\big)''+f''(\theta)(q^f_1-\alpha \theta')=\lambda_1^f q^f_1,\ \ (q^f_1-\alpha \theta')'\bigg|^{\frac{\delta}{\ve}}_{-\frac{\delta}{\ve}}=-\alpha \theta''\bigg|^{\frac{\delta}{\ve}}_{-\frac{\delta}{\ve}}.\nonumber
\end{align}
Multiplying the above equation by $q^f_1-\alpha \theta'$, integrating by parts and using (\ref{eigenfunc:f}), we have
\begin{align}
\int_{-\frac{\delta}{\ve}}^{\frac{\delta}{\ve}}\big|\big(q^f_1-\alpha \theta'\big)'\big|^2dz
&=-\int_{-\frac{\delta}{\ve}}^{\frac{\delta}{\ve}}f''(\theta)(q^f_1-\alpha \theta')^2dz
+\lambda_1^f\int_{-\frac{\delta}{\ve}}^{\frac{\delta}{\ve}}q^f_1(q^f_1-\alpha \theta')dz
\nonumber\\&\quad-\alpha^2\theta''(\frac{\delta}{\ve})\theta'(\frac{\delta}{\ve})+\alpha^2\theta''(-\frac{\delta}{\ve})\theta'(-\frac{\delta}{\ve})
\nonumber\\&=O(e^{-\frac{C}{\ve}}).\label{eigenfunction derivative}
\end{align}
It follows  from (\ref{eigenfunction derivative}) that
\begin{align}
\ve\gamma\int_{-\frac{\delta}{\ve}}^{\frac{\delta}{\ve}}\hat{v}(q_1^f-\alpha \theta')'J_{r_t}J^{-1}dz&\leq C\ve\gamma\bigg(\int_{-\frac{\delta}{\ve}}^{\frac{\delta}{\ve}}|\hat{v}|^2dz\bigg)^{\frac12}\bigg(
\int_{-\frac{\delta}{\ve}}^{\frac{\delta}{\ve}}|(q_1^f-\alpha \theta')'|^2dz\bigg)^{\frac12}
\nonumber\\&\leq  O(e^{-\frac{C}{\ve}})\bigg(\int_{-\frac{\delta}{\ve}}^{\frac{\delta}{\ve}}|\hat{v}|^2dz\bigg)
\leq  O(e^{-\frac{C}{\ve}})\int_{-\frac{\delta}{\ve}}^{\frac{\delta}{\ve}}v^2dz.\label{Jacobian inequality-5}
\end{align}

And  from (\ref{Jacobic derivative}) one has
\begin{align}
\ve\alpha\gamma\bigg|\int_{-\frac{\delta}{\ve}}^{\frac{\delta}{\ve}}\hat{v}\theta''J_{r_t}J^{-1}dz\bigg|
&\leq C\ve^2\alpha\gamma\int_{-\frac{\delta}{\ve}}^{\frac{\delta}{\ve}}|\hat{v}\theta''z|dz\nonumber\\
&\leq C\ve^2\alpha\gamma\bigg(\int_{-\frac{\delta}{\ve}}^{\frac{\delta}{\ve}}|\hat{v}|^2dz\bigg)^{\frac{1}{2}}
\bigg(\int_{-\frac{\delta}{\ve}}^{\frac{\delta}{\ve}}|\theta''z|^2dz\bigg)^{\frac{1}{2}}
\nonumber\\
&\leq C\ve^2\int_{-\frac{\delta}{\ve}}^{\frac{\delta}{\ve}}v^2dz.\label{Jacobian inequality-6}
\end{align}

Using (\ref{eigen:f-2}) we can arrive at
\begin{align}
\int_{-\frac{\delta}{\ve}}^{\frac{\delta}{\ve}}\big|\partial_zp_1\big|^2dz&= \int_{-\frac{\delta}{\ve}}^{\frac{\delta}{\ve}}\bigg(\big| ( \partial_zp_1\big|^2+f''(\theta)({p_1})^2\bigg)dz-\int_{-\frac{\delta}{\ve}}^{\frac{\delta}{\ve}}f''(\theta)({p_1})^2dz
\nonumber\\&\leq C\int_{-\frac{\delta}{\ve}}^{\frac{\delta}{\ve}}\bigg(\big|\partial_zp_1\big|^2+f''(\theta)({p_1})^2\bigg)dz
\nonumber\\&=C\int_{-\frac{\delta}{\ve}}^{\frac{\delta}{\ve}}\bigg(\partial_z \hat{v})^2+f''(\theta)\hat{v}^2\bigg)dz-C\lambda_1^f\gamma^2
\nonumber\\&\leq C\int_{-\frac{\delta}{\ve}}^{\frac{\delta}{\ve}}\bigg((\partial_z \hat{v})^2+f''(\theta)\hat{v}^2\bigg)dz+Ce^{-\frac{C_2}{\varepsilon}}\int_{-\frac{\delta}{\ve}}^{\frac{\delta}{\ve}}\hat{v}^2dz.\label{ortho derivative}\end{align}
By (\ref{ortho derivative}) and the Young's inequality, one has
\begin{align}
\bigg|-\ve\int_{-\frac{\delta}{\ve}}^{\frac{\delta}{\ve}}\hat{v}\partial_z{p_1}J_{r_t}J^{-1}dz\bigg|&\leq C\ve \bigg(\int_{-\frac{\delta}{\ve}}^{\frac{\delta}{\ve}}v^2dz \bigg)^{\frac{1}{2}}\bigg(\int_{-\frac{\delta}{\ve}}^{\frac{\delta}{\ve}}(\partial_z{p_1})^2dz
\bigg)^{\frac{1}{2}}
\nonumber\\&\leq C\ve^2\int_{-\frac{\delta}{\ve}}^{\frac{\delta}{\ve}}v^2dz+\frac{1}{4}\int_{-\frac{\delta}{\ve}}^{\frac{\delta}{\ve}}\big((\partial_z\hat{v})^2+f''(\theta)\hat{v}^2\big)dz.\nonumber
\end{align}
Thus
\begin{align}
-\ve\int_{-\frac{\delta}{\ve}}^{\frac{\delta}{\ve}}\hat{v}\partial_z{p_1}J_{r_t}J^{-1}dz\geq
-C\ve^2\int_{-\frac{\delta}{\ve}}^{\frac{\delta}{\ve}}v^2dz-\frac{1}{4}\int_{-\frac{\delta}{\ve}}^{\frac{\delta}{\ve}}\big((\partial_z\hat{v})^2+f''(\theta)\hat{v}^2\big)dz.\label{Jacobian inequality-7}
\end{align}

Substituting (\ref{Jacobian inequality-5}), (\ref{Jacobian inequality-6}) and (\ref{Jacobian inequality-7}) into  (\ref{Jacobian inequality-4}) we have
\begin{align}
-\ve\int_{-\frac{\delta}{\ve}}^{\frac{\delta}{\ve}}\hat{v}\partial_z{\hat{v}}J_{r_t}J^{-1}dz\geq
-C\ve^2\int_{-\frac{\delta}{\ve}}^{\frac{\delta}{\ve}}v^2dz-\frac{1}{4}\int_{-\frac{\delta}{\ve}}^{\frac{\delta}{\ve}}\big((\partial_z\hat{v})^2+f''(\theta)\hat{v}^2\big)dz\nonumber
\end{align}
which together with  (\ref{Jacobian inequality -3}) leads to (\ref{Jacobic transform}).

Hence the proof of  (\ref{Jacobic transform}) is finished.\end{proof}
\section*{Acknowledgments}
The authors would like to thank Professor Z. Zhang in Peking University for helpful discussions. M. Fei is partly supported by NSF of China under Grant 11301005 and AHNSF grant 1608085MA13. W. Wang is partly supported by NSF of China under Grant 11501502 and ``the Fundamental Research Funds for the Central Universities" No. 2016QNA3004.

\end{document}